\documentclass[12pt a4paper]{amsart}

\usepackage{stmaryrd}
\usepackage{tikz}
\usepackage{tikz-cd}

\usepackage{amscd}
\usepackage[section]{placeins}

\usepackage{amsmath, amsfonts,amssymb,amsthm,mathrsfs,xypic}
\usepackage{graphicx}
\usepackage{pstricks}
\usepackage{lscape}
\DeclareFontFamily{U}{mathx}{\hyphenchar\font45}
\DeclareFontShape{U}{mathx}{m}{n}{
      <5> <6> <7> <8> <9> <10>
      <10.95> <12> <14.4> <17.28> <20.74> <24.88>
      mathx10
      }{}
\DeclareSymbolFont{mathx}{U}{mathx}{m}{n}
\DeclareFontSubstitution{U}{mathx}{m}{n}
\DeclareMathAccent{\widecheck}{0}{mathx}{"71}
\DeclareMathAccent{\wideparen}{0}{mathx}{"75}

\xyoption{curve}

\usepackage[all]{xy}
\numberwithin{equation}{subsection}
\usepackage{microtype}

\numberwithin{equation}{section}

\theoremstyle{plain}
\newtheorem{theorem}[equation]{Theorem}
\newtheorem{proposition}[equation]{Proposition}
\newtheorem{lemma}[equation]{Lemma}
\newtheorem{corollary}[equation]{Corollary}

\theoremstyle{definition}
\newtheorem{definition}[equation]{Definition}
\newtheorem{example}[equation]{Example}

\theoremstyle{remark}
\newtheorem{remark}[equation]{Remark}
\newcommand{\Coker}{\operatorname{Coker}}
\newcommand{\Ext}{\operatorname{Ext}}

\newcommand{\Hom}{\operatorname{Hom}}
\newcommand{\Ker}{\operatorname{Ker}}

\newcommand{\R}{\mathrm{R}}
\newcommand{\rmL}{\mathrm{L}}

\newcommand{\ul}{\underline}
\newcommand{\xrto}{\xrightarrow}
\def\A{\mathcal A}

\def\C{\mathcal C}

\def\E{\mathcal E}
\def\F{\mathcal F}

\def\I{\mathcal I}

\def\T{\mathcal T}
\def\H{\mathcal H}

\def\X{\mathcal X}
\def\Y{\mathcal Y}

\begin{document}

\title[\tiny{The triangulation of subfactors}]{The triangulation of the subfactor categories of additive categories with suspensions}
\author [\tiny{Zhi-Wei Li}] {Zhi-Wei Li}
%\thanks{$^*$ Corresponding to: }

\date{\today}
\thanks{The author was supported by National Natural Science Foundation
of China (No.s 11671174 and 11571329).}

\email{zhiweili@jsnu.edu.cn}
\maketitle

%\address{}

\begin{center}
\tiny{School of Mathematics and Statistics, \ \ Jiangsu Normal University \\
Xuzhou 221116, Jiangsu, PR China.}
\end{center}
\begin{abstract}
We provide a framework to triangulate subfactor categories of additive categories with additive endofunctors. It is proved that such a framework is sufficiently flexible to cover many instances in algebra and geometry where abelian, exact and triangulated subfactor categories are constructed. As an application, we show that Iyama-Yoshino triangulated subfactor categories can be modeled.
\end{abstract}

\setcounter{tocdepth}{1}
%\tableofcontents

\section{Introduction}
Ever since the triangulated categories were introduced by Verdier \cite{Verdier} in 1963, and the exact categories were introduced by Quillen \cite{Quillen73} in 1973 as a generalization of abelian categories defined by Buchsbaum \cite{Buchsbaum} and Grothendieck \cite{Grothendieck} in the late 1950's,  they have been two powerful tools in algebra and geometry. One of the most important bridges between exact categories and triangulated categories is the subfactor or stable categories. It is well known that the stable category of a Frobenius exact category modulo the projective-injectives are triangulated categories \cite{Happel88}. Meanwhile, with the development of the theory of cluster algebras and cluster categories, many examples in the opposite direction have been constructed \cite{Buan-Marsh-Reiten, Keller/Reiten, Ringel-Zhang, Koenig/Zhu}. Quite recently, the work of Iyama and Yoshino \cite{Iyama-Yoshino} and Kussin, Lenzing and Meltzer \cite{Kussin/Lenzing/Meltz}  show that triangulated and exact structures are even closed under taking subfactor categories in some cases.

This paper is aimed to give a framework to triangulate the subfactor categories of additive categories with suspensions in the sense of \cite{Heller68} which include both exact categories and triangulated categories. It is based on the notions of {\it partial one-sided triangulated categories} which unify and cover the instances mentioned above where abelian, exact and triangulated subfactor categories are constructed. %However, our starting point to introduce one-sided pseudo-triangulated categories is to develop a homotopy theory of additive categories with suspensions not depending on Quillen model structures \cite{ZWLi1}.

We now give some details about our results. Roughly speaking, a partial right triangulated category consists of an additive category $\A$ endowed with an additive endofunctor $\Sigma$, two additive subcategories $\X, \C$ of $\A$ together with a class $\R(\C)$ of $\Sigma$-sequences, which satisfy similar axioms of a {\it right triangulated category} \cite{Beligiannis/Marmaridis94} except the rotation axiom. Right triangulated categories and exact categories are examples of partial right triangulated categories. The precise definition of a partial right triangulated category can be found in Definition \ref{def:prtc}.  A partial left triangulated category is defined dually. Our first main result is the following:
\vskip5pt
\noindent{\bf Theorem} (\ref{thm:main}). {\it $(\mathrm i)$ \ If $(\A, \Sigma, \R(\C), \X)$ is a partial right triangulated category, then the subfactor $\C/\X$ has a right triangulated structure induced by $\R(\C)$.
	
	$(\mathrm{ii})$ \ If $(\A, \Omega, \rmL(\C), \X)$ is a partial left triangulated category, then the subfactor $\C/\X$ has a left triangulated structure induced by $\rmL(\C)$.}
\vskip5pt
\noindent This result covers many existed constructions of one-sided triangulated categories in various settings \cite[Theorem 2.12]{Beligiannis/Marmaridis94}, \cite[Theorem 7.1]{Beligiannis2000}, \cite[Theorem 3.9]{Liu-Zhu} and \cite[Theorem 3.7]{ZWLi}.

\vskip5pt

In general, it is rare that the exact structure of an exact category could be inherited by its factor category. A surprising example was observed by Kussin, Lenzing and Meltzer in their study of weighted projective lines \cite{Kussin/Lenzing/Meltz}. We use our first result to give a general framework for the construction of exact factor categories.
\vskip5pt
\noindent{\bf Theorem} (\ref{thm:sexact}). {\it  Let $(\A, \E)$ be an exact category and $\X$ an additive subcategory of $\A$. Assume that for each $A\in \A$, there are conflations $X_1\to X_0\stackrel{p}\to A $ and $ A\stackrel{i}\to X^0\to X^1$ such that $p$ is an $\X$-precover, $i$ is an $\X$-preenvelope and $X_1, X^1\in \X$, then every morphism in $\A/\X$ has a kernel and a cokernel and $\A/\X$ has an induced exact structure by $\E$.}

\vskip5pt
Our third main result is about the construction of triangulated subfactor categories. The result is based on the notion of a {\it partial triangulated category} which is an additive category with compatible partial right and left triangulated structures. It covers many existed constructions of triangulated subfactor categories \cite[Theorem 2.6]{Happel88}, \cite[Theorem 4.2]{Iyama-Yoshino}, \cite[Theorem 3.3 (i)]{Beligiannis13} and \cite[Theorem 6.17]{Nakaoka}:

\vskip5pt
\noindent{\bf Theorem} (\ref{thm:ptc}). {\it Let $(\A, \Omega, \Sigma, \rmL(\C), \R(\C), \X)$ be a partial triangulated category. Then the subfactor $\C/\X$ is a triangulated category.}

\vskip5pt
\noindent We use this result to give a Quillen model structure of Iyama-Yoshino triangulated subfactors in Corollary \ref{cor:IYmodel}.

We now sketch the contents of the paper. In Section 2, we introduce the new notion of partial one-sided triangulated categories and give some examples. Section 3 is devoted to triangulating the subfactor categories arising from partial one-sided triangulated categories. In Section 4, we give various examples of partial one-sided triangulated categories in additive, exact and triangulated categories. We construct, in Section 5, the exact subfactor categories from exact categories and abelian subfactor categories from triangulated categories. In Section 6, we introduce the notion of a partial triangulated category, prove our third main result and then model Iyama-Yoshino triangulated subfactor categories.

\vskip5pt

Throughout this paper, unless otherwise stated, that all subcategories of additive categories considered are full, closed under isomorphisms, all functors between additive categories are assumed to be additive.

\subsection*{Acknowledgements} I would like to thank Henning Krause, Xiao-Wu Chen, Yu Ye and Greg Stevenson for
their helpful discussions and suggestions. %I am grateful to Hongxing Chen, Jiangsheng Hu, Wei Hu, Clause Michael Ringel, Jiaqun Wei, Changchang Xi, Baolin Xiong, Guodong Zhou and Yu Zhou for their encouragements.
I would especially like to thank Yan Lu for her translating \cite{Keller/Vossieck87} into English.
\section{Partial one-sided triangulated categories}

In this section we recall the definition of a right triangulated category and introduce the new notion of a partial right triangulated category.

\subsection*{Stable categories of additive categories} Let $\C$ be an additive category and $\X$ an additive subcategory of $\C$. Given two morphisms $f,f'\colon A\to B$
in $\C$, we say that $f$ is {\it stably equivalent} to $f'$, written $f\sim f'$, if $f-f'$ factors through $\X$ (that is, there exists some object $X\in \X$ such that
there are two morphisms $u\colon X\to B$ and $v\colon A\to X $ satisfying $f-f'=u\circ v$). We use $\ul{f}$ to denote the stable equivalence class of $f$.
It is well known that stable equivalence is an equivalence relation which is compatible with composition. That is, if $f\sim f'$, then $f\circ k\sim f'\circ k$ and
 $h\circ f\sim h\circ f'$ whenever the compositions make sense. The {\it stable} or {\it factor category} $\C/\X$ is the category whose objects are the objects of $\C$, and whose morphisms are the stable equivalence classes of $\C$. Recall that the stable category $\C/\X$ is an additive category.

\subsection*{Right triangulated categories} If $\H$ is an arbitrary category endowed with a functor $\Sigma\colon \H\to \H$ (such a category is called a {\it category with suspension} \cite{Heller68}), following \cite{Keller90}, a sequence of the form
$$A\stackrel{f}\to B\stackrel{g}\to C \stackrel{h}\to \Sigma(A)$$
in $\H$ will be called a {\it right $\Sigma$-sequence}. A {\it morphism of right $\Sigma$-sequences} is given by a commutative diagram
	\begin{equation*}
	\xy\xymatrixcolsep{2pc}\xymatrix@C14pt@R14pt{  A\ar[r]^f \ar[d]_\alpha & B\ar[r]^g\ar[d]^\beta& C\ar[d]^\gamma\ar[r]^{h} & \Sigma(A)\ar[d]^{\Sigma(\alpha)}
		\\
		A'\ar[r]^{f'} & B'\ar[r]^{g'}& C'\ar[r]^{h'} & \Sigma(A')}
	\endxy
	\end{equation*}
The composition is the obvious one. Dually we can define the notion of a {\it left $\Sigma$-sequence} in $\H$.

We recall the definition of a right triangulated category; see the dual of \cite[Definition 2.2]{Beligiannis/Marmaridis94} and \cite[Definition 1.1]{ABM}.
\begin{definition} \label{def:rtricat}\ Let $\T$ be an additive category endowed with an additive endofunctor $\Sigma$. Let $\Delta$ be a class of right $\Sigma$-sequences called {\it right triangles}. The triple $(\T, \Sigma, \Delta)$ is called a {\it right triangulated category} if $\Delta$ is closed under isomorphisms and the following four axioms hold:

	(RT1) \ For any morphism $f\colon A\to B$, there is a right $\Sigma$-sequence $A\xrto{f} B\to C\to \Sigma(A)$ in $\Delta$. For any object $A\in \T$, the right $\Sigma$-sequence $0\to A\xrto{1_A} A\to 0$ is in $\Delta$.
	
	(RT2)  (Rotation axiom) \ If $A\xrto{f} B\xrto{g} C\xrto{h} \Sigma(A)$ is a right triangle, so is $B\xrto{g} C\xrto{h} \Sigma(A) \xrto{-\Sigma(f)} \Sigma(B)$.

	(RT3) \ If the rows of the following diagram are right triangles and the leftmost square is commutative, then there is a morphism $\gamma\colon C\to C'$ making the whole diagram commutative:	
\[\xy\xymatrixcolsep{2pc}\xymatrix@C14pt@R14pt{A\ar[r]^f\ar[d]_{\alpha}&B \ar[r]^g\ar[d]^\beta&C\ar[r]^-h\ar@{.>}[d]^\gamma & \Sigma(A)\ar[d]^{\Sigma(\alpha)}\\
A'\ar[r]^{f'}&B'\ar[r]^{g'}&C'\ar[r]^-{h'}&\Sigma(A')}
	\endxy\]

(RT4) \ For any three right triangles: $A\xrto{f} B\xrto{l} C'\to \Sigma(A)$, $B\xrto{g} C\xrto{h} A'\xrto{j} \Sigma(B)$ and $A\xrto{g\circ f} C\to B'\to \Sigma{A}$, there is a commutative diagram	
\[\xy\xymatrixcolsep{2pc}\xymatrix@C14pt@R14pt{
A\ar[r]^f\ar@{=}[d]&B\ar[r]^l\ar[d]^g&C'\ar[r]\ar@{.>}[d]& \Sigma(A)\ar@{=}[d]\\
A\ar[r]^{g\circ f}&C\ar[r]\ar[d]^{h}&B'\ar[r]\ar@{.>}[d]&\Sigma(A)\ar[d]^{\Sigma(f)}\\
&A'\ar@{=}[r]\ar[d]^{j}&A'\ar[r]^-{j}\ar[d]&\Sigma(B)\\
&\Sigma(B)\ar[r]^{\Sigma(l)}&\Sigma(C')}
	\endxy\]
	such that the second column from the right is a right triangle.
\end{definition}
The notion of a {\it left triangulated category} is defined dually.
\vskip5pt
 One-sided triangulated categories as a generalization of triangulated categories arise naturally in the study of homotopy theories \cite{Heller60, Quillen67, Heller68, Brown73} and derived categories \cite{Keller/Vossieck87, Keller91}. When the endofunctor $\Sigma$ is an auto-equivalence, a right triangulated category $(\T, \Sigma, \Delta)$ is a triangulated category in the sense of \cite{Verdier}.
\subsection*{Partial right triangulated categories}
Let $\A$ be an additive category endowed with an additive endofunctor $\Sigma\colon \A\to \A$. We use $\X\subseteq \C$ to denote that $\X, \C$ are additive subcategories of $\A$ such that $\X$ is a subcategory of $\C$. A morphism $f\colon A\to B$ in $\A$ is said to be an {\it $\X$-monic} if the induced morphism $f^*=\Hom_\A(f, \X)\colon \Hom_\A(B, \X)\to \Hom_\A(A, \X)$ is surjective. The notion of an {\it$\X$-epic} is defined dually. Recall that a morphism $f\colon A\to X$ in $\A$ is called an {\it $\X$-preenvelope} (also called a {\it left $\X$-approximation} of $A$ in some literatures) if $f$ is an $\X$-monic and $X\in \X$. Dually a morphism $g\colon X\to A$ is called an {\it $\X$-precover} if $g$ is an $\X$-epic and $X\in \X$.

 A right $\Sigma$-sequence $A\stackrel{f}\to B\stackrel{g}\to C\stackrel{h}\to \Sigma(A)$
in $\A$ is called a {\it right $\C$-sequence} if $C\in \C$, $g$ is a {\it weak cokernel} of $f$ (i.e. the induced sequence
$\Hom_\A(C, \A)\to \Hom_\A(B, \A)\to \Hom_\A(A, \A)$ is exact) and $h$ is a weak cokernel of $g$.

Dually, a left $\Sigma$-sequence $\Sigma(B)\xrto{u} K\xrto{v} A\xrto{f} B$ is called a {\it left $\C$-sequence} if $K\in \C$ and $v$ is a {\it weak kernel} of $f$ and $u$ is a weak kernel of $v$.
\vskip5pt
Now we are in the position to introduce the new concept of partial one-sided triangulated categories.
\begin{definition}\label{def:prtc}\ Let $\A$ be an additive category endowed with an additive endofunctor $\Sigma$. Let $\X\subseteq \C$ be two additive subcategories of $\A$ and $\R(\C)$ a class  of right $\C$-sequences (called {\it right $\C$-triangles}).  The pair $(\R(\C),\X)$ is said to be a {\it partial right triangulated structure on $\A$} if $\R(\C)$ is closed under isomorphisms and finite direct sums and the following axioms hold:

(PRT1) (i) \ For each $A\in \C$, there is a right $\C$-triangle $A\xrto{i} X\to U\to \Sigma(A)$ with $i$ an $\X$-preenvelope.

 (ii) \ For each morphism $f\colon A\to B$ in $\C$, $A\xrightarrow{\left(\begin{smallmatrix}
	1 \\
	f
	\end{smallmatrix}\right)} A\oplus B \xrightarrow{(f, -1)} B\xrto{0} \Sigma(A)$ is in $\R(\C)$.

(iii) \ If $A\xrto{i} X\to U\to \Sigma(A)$ is in $\R(\C)$ with $i$ an $\X$-preenvelope in $\C$, then for any morphism $f\colon \ A\to B$ in $\C$, there is a right $\C$-triangle $A\xrightarrow{\left(\begin{smallmatrix}
	i \\
	f
	\end{smallmatrix}\right)} X\oplus B \to N\to \Sigma(A)$.

(PRT2) \ For any commutative diagram of right $\C$-triangles
\[\xy\xymatrixcolsep{2pc}\xymatrix@C14pt@R14pt{
A\ar[r]^f\ar[d]_{\alpha}&B\ar[r]\ar[d]&C\ar[r]\ar[d]^\gamma& \Sigma(A)\ar[d]^{\Sigma(\alpha)}\\
A'\ar[r]&X\ar[r]^s&U\ar[r]&\Sigma(A')}
\endxy\]
	with $X\in \X$, if $\alpha$ factors through $f$, then $\gamma$ factors through $s$.

(PRT3) \ If the rows of the following diagram are in $\R(\C)$ and the leftmost square is commutative, then there is a morphism $\gamma\colon C\to C'$ making the whole diagram commutative:
		\[\xy\xymatrixcolsep{2pc}\xymatrix@C14pt@R14pt{A\ar[r]^f\ar[d]_{\alpha}&B \ar[r]^g\ar[d]^\beta&C\ar[r]^-h\ar@{.>}[d]^\gamma & \Sigma(A)\ar[d]^{\Sigma(\alpha)}\\
A'\ar[r]^{f'}&B'\ar[r]^{g'}&C'\ar[r]^-{h'}&\Sigma(A')}
	\endxy\]

(PRT4) \ If $A\xrto{f} B\xrto{l} C'\to \Sigma(A)$, $B\xrto{g} C \xrto{h} A'\xrto{j} \Sigma(B)$ and $A\xrightarrow{g\circ f} C \xrto{k} B'\to \Sigma(A)$ are in $\R(\C)$ such that $f, g$ are $\X$-monics in $\C$, then there is a commutative diagram
\[\xy\xymatrixcolsep{2pc}\xymatrix@C14pt@R14pt{
A\ar[r]^f\ar@{=}[d]&B\ar[r]^l\ar[d]^g&C'\ar[r]\ar@{.>}[d]^r& \Sigma(A)\ar@{=}[d]\\
A\ar[r]^{g\circ f}&C\ar[r]\ar[d]^{h}&B'\ar[r]\ar@{.>}[d]&\Sigma(A)\ar[d]^{\Sigma(f)}\\
&A'\ar@{=}[r]\ar[d]^{j}&A'\ar[r]^-{j}\ar[d]&\Sigma(B)\\
&\Sigma(B)\ar[r]^{\Sigma(l)}&\Sigma(C')}
	\endxy\]
such that the second column from the right is in $\R(\C)$ with $r$ an $\X$-monic.
\end{definition}

The quadruple $(\A, \Sigma, \R(\C), \X))$ is said to be a {\it partial right triangulated category} if $(\R(\C), \X)$ is a partial right triangulated structure on $\A$.

Dually, we can define the notion of a {\it partial left triangulated category}.

\begin{example} \label{babyexam} (i) \ If $(\T, \Sigma, \Delta)$ is a right triangulated category, take $\X=0, \C=\T$ and $\R(\T)=\Delta$, then $(\T, \Sigma, \R(\T),0)$ is a partial right triangulated category. Dually, if $(\T, \Omega, \nabla)$ is a left triangulated category, then $(\T, \Omega, \rmL(\T)=\nabla, 0)$ is a partial left triangulated category.
\vskip5pt
(ii) Let $(\A, \mathcal{E})$ be an exact category and $\rmL(\A)=\{ 0\to A\to B\to C \ | \ A\to B\to C \in \mathcal{E}\}$. Then $(\A, 0, \rmL(\A), \A)$ is a partial left triangulated category. Denote by $\mathcal{P}$ the subcategory of projectives in $\A$, if $\A$ has enough projectives, i.e. each object in $\A$ has a $\mathcal{P}$-precover, then $(\A, 0, \rmL(\A), \mathcal{P})$ is a partial left triangulated category. Dually, let $\R(\A)=\{ A\to B\to C \to 0\ | \ A\to B\to C \in\mathcal{E}\}$, then $(\A, 0, \R(\A), \A)$ is a partial right triangulated category, and if $\A$ has enough injectives, then $(\A,  0, \R(\A),\I)$ is a partial right triangulated category, where $\I$ is the subcategory of injectives in $\A$.
\end{example}

Next we fix a partial right triangulated category $(\A, \Sigma, \R(\C), \X)$.

\begin{lemma} \label{lemma:uniqueness}\ 	$({\mathrm i})$  For a given morphism $f\colon A\to B$ in $\A$, if $A\xrto{f} B\xrto{g} C\xrto{h} \Sigma(A)$ is in $\R(\C)$, then it is unique up to isomorphism.	
	
	$(\mathrm{ii})$ \ Let $A\xrto{r} X\xrto{s} U\xrto{t} \Sigma(A)$ and $A\xrto{r'} X'\xrto{s'} U'\xrto{t'} \Sigma(A)$ be two right $\C$-sequences in $\R(\C)$ such that $r,r'$ are $\X$-preenvelopes in $\C$. Then $U\cong U'$ in $\C/\X$.
\end{lemma}	

\begin{proof} \	(i) \ We first show that any morphism $m: C\to C $ in the following commutative diagram is an isomorphism:
\[\xy\xymatrixcolsep{2pc}\xymatrix@C14pt@R14pt{
A\ar[r]^f\ar@{=}[d]&B\ar[r]^g\ar@{=}[d]&C\ar[r]^-h\ar[d]^m& \Sigma(A)\ar@{=}[d]\\
A\ar[r]^f&B\ar[r]^g&C\ar[r]^-h&\Sigma(A)}
\endxy\]
		In fact, since $(m-1_C)\circ g=m\circ  g-g=0$ and $h$ is a weak cokernel of $g$, there is a morphism $l: \Sigma(A)\to C$ such that $m-1_C=l\circ h$. Then $(m-1_C)^2=l\circ h\circ (m-1_C)=l\circ (h\circ m -h)=l\circ 0=0$. Thus $m$ is an isomorphism.
	
	If there is another right $\C$-triangle $A\xrto{f} B\xrto{g'} C'\xrto{h'} \Sigma(A)$, then by (PRT3), there are morphisms $u\colon C\to C'$ and $u'\colon C'\to C$ making the following diagram commutative: 	
	\[\xy\xymatrixcolsep{2pc}\xymatrix@C14pt@R14pt{
A\ar[r]^f\ar@{=}[d]&B\ar[r]^g\ar@{=}[d]&C\ar[r]^-h\ar[d]^u& \Sigma(A)\ar@{=}[d]\\
A\ar[r]^f\ar@{=}[d]&B\ar[r]^{g'}\ar@{=}[d]&C'\ar[r]^-{h'}\ar[d]^{u'}&\Sigma(A)\ar@{=}[d]\\
A\ar[r]^f&B\ar[r]^g&C\ar[r]^-h&\Sigma(A)}
\endxy\]

	Then both $u'\circ u$ and $u\circ u'$ are isomorphisms by the previous proof, thus $u$ is an isomorphism.

 (ii) \ There is a commutative diagram of right $\C$-triangles:
	 	\[\xy\xymatrixcolsep{2pc}\xymatrix@C14pt@R14pt{
A\ar[r]^r\ar@{=}[d]&X\ar[r]^s\ar[d]^v&U\ar[r]^-t\ar[d]^\delta& \Sigma(A)\ar@{=}[d]\\
A\ar[r]^{r'}\ar@{=}[d]&X'\ar[r]^{s'}\ar[d]^{v'}&U'\ar[r]^-{t'}\ar[d]^{\delta'}&\Sigma(A)\ar@{=}[d]\\
A\ar[r]^r&X\ar[r]^s&U\ar[r]^-t&\Sigma(A)}
\endxy\]
 where the existence of $v$ and $v'$ is since both $r$ and $r'$ are $\X$-monics, and the existence of $\delta$ and $\delta'$ is by (PRT3). Then we have the following commutative diagram
	 	 	\[\xy\xymatrixcolsep{2pc}\xymatrix@C24pt@R14pt{
A\ar[r]^r\ar[d]_0&X\ar[r]^s\ar[d]^{v'\circ v-1_X}&U\ar[r]^-t\ar[d]^{\delta'\circ \delta-1_U}& \Sigma(A)\ar[d]^0\\
A\ar[r]^r&X\ar[r]^s&U\ar[r]^-t&\Sigma(A)}
\endxy\]
	 Thus $\delta'\circ \delta-1_U$ factors through $s$ by (PRT2) since $0: A\to A$ factors through $r$. In other words $\ul{\delta}'\circ \ul{\delta}=\ul{1}_U$. Similarly, we can show $\ul{\delta}\circ \ul{\delta}'=\ul{1}_{U'}$. Thus $U$ is isomorphic to $U'$ in $\C/\X$.\end{proof}

\begin{lemma} \label{lemma:notation}
	Let $A\xrto{f} B\xrto{g} C\xrto{h} \Sigma(A)$ and $D\xrto{r} X\xrto{s} U\xrto{t}\Sigma(D)$ be two right $\C$-sequences in $\R(\C)$ with $f$ an $\X$-monic and $X\in \X$. If there is a morphism $\alpha \colon A\to D$, then there is a morphism $(\alpha, \beta, \gamma)$ of right $\C$-triangles:
		\[\xy\xymatrixcolsep{2pc}\xymatrix@C14pt@R14pt{
A\ar[r]^{f}\ar[d]_\alpha & B\ar[r]^g\ar@{.>}[d]^\beta&C\ar[r]^-h\ar@{.>}[d]^\gamma& \Sigma(A)\ar[d]^{\Sigma(\alpha)}\\
D\ar[r]^r&X\ar[r]^{s}&U\ar[r]^-{t}&\Sigma(D)}
\endxy\]
	and $\ul{\gamma}$ is uniquely determined by $\ul{\alpha}$ in $\C/\X$.
	\end{lemma}
\begin{proof}
	The existence of $\beta$ is since $f$ is an $\X$-monic and $X\in \X$. The existence of $\gamma$ is by (PRT3). Assume that $\ul{\alpha}=0$ in $\C/\X$, i.e., $\alpha$ factors through some object in $\X$, then it must factor through $f$ since $f$ is an $\X$-monic. Thus $\gamma$ factors through $s$ by (PRT2), in other words $\ul{\gamma}=0$ in $\C/\X$.
	\end{proof}

\subsection*{Suspension functor of the subfactor category of a partial right triangulated category} If $(\A,\Sigma, \R(\C), \X)$ is a partial right triangulated category, then by (PRT1)(i), for each object $A\in \C$, we can {\it fix} a right $\C$-triangle $A\xrto{i^A} X^A\xrto{p^A} U^A\xrto{q^A} \Sigma(A)$ with $i^A$ an $\X$-preenvelope. Then we can define a functor $\Sigma^\X\colon \C/\X\to \C/\X$ by sending an object $A$ to $U^A$ and a morphism $\ul{f}\colon A\to B$ to $\ul{\kappa}^f$, where $\kappa^f$ satisfies the following commutative diagram:
\begin{equation}\label{kappaf}
\xy\xymatrixcolsep{2pc}\xymatrix@C14pt@R14pt{
A\ar[r]^{i^A}\ar[d]_f&X^A\ar[r]^{p^A}\ar[d]^{\sigma^f}&U^A\ar[r]^-{q^A}\ar[d]^{\kappa^f}& \Sigma(A)\ar[d]^{\Sigma(f)}\\
B\ar[r]^{i^B}&X^B\ar[r]^{p^B}&U^B\ar[r]^-{q^B}&\Sigma(B)}
\endxy
\end{equation}
By Lemma \ref{lemma:notation}, $\ul{\kappa}^f$ is uniquely determined by $\ul{f}$ and thus $\Sigma^\X$ is well defined.

\section{From partial one-sided triangulated categories to one-sided triangulated categories}

In this section, we fix a partial right triangulated category $(\A, \Sigma, \R(\C), \X)$ and construct the right triangulated structure on the subfactor category $\C/\X$.
\subsection*{Right triangulated structures induced from partial right triangulated structures}
 If $A\stackrel{f}\to B\stackrel{g}\to C\stackrel{h}\to \Sigma(A) $ is a right $\C$-triangle with $f$ an $\X$-monic in $\C$, by Lemma \ref{lemma:notation}, we have a commutative diagram
\begin{equation}\label{xif}
\xy\xymatrixcolsep{2pc}\xymatrix@C16pt@R14pt{
A\ar[r]^f\ar@{=}[d]&B\ar[r]^g\ar[d]^{\delta^f}&C\ar[r]^-h\ar[d]^{\xi(f,g)}& \Sigma(A)\ar@{=}[d]\\
A\ar[r]^{i^A}&X^A\ar[r]^{p^A}&U^A\ar[r]^-{q^A}&\Sigma(A)}
\endxy
\end{equation}
and the residue class of $\ul{\xi}(f,g)$ is uniquely determined by the upper row. Thus there is an induced right $\Sigma^\X$-sequence $A\stackrel{\ul{f}}\to B\stackrel{\ul{g}}\to C\xrightarrow{\ul{\xi}(f,g)}\Sigma^\X(A)$ in $\C/\X$ which is said to be a {\it standard right triangle}. We use $\Delta^\X$ to denote the class of right $\Sigma^\X$-sequences in $\C/\X$ which are isomorphic to standard right triangles. The elements in $\Delta^\X$ are called {\it distinguished right triangles}.

Dually, if $(\A, \Omega, \rmL(\C), \Y)$ is a partial left triangulated category, we can construct an additive endofunctor $\Omega_\Y\colon \C/\Y\to \C/\Y$, and define the corresponding {\it standard left triangles}. We use $\nabla_\Y$ to denote the class of left $\Omega_\Y$-sequences in $\C/\Y$ which are isomorphic to standard left triangles.

\vskip5pt
We have the following theorem
\begin{theorem} \label{thm:main}  $(\rm i)$ \ If $(\A, \Sigma, \R(\C), \X)$ is a partial right triangulated category, then $(\C/\X, \Sigma^\X, \Delta^\X)$ is a right triangulated category.
	
	$(\mathrm{ii})$ \ If $(\A, \Omega, \rmL(\C), \Y)$ is a partial left triangulated category, then $(\C/\Y, \Sigma_\Y, \nabla_\Y)$ is a left triangulated category.
\end{theorem}

Before the proof we need two Lemmas.

\begin{lemma} \label{lem:octahedral} Assume that we have a commutative diagram of right $\C$-triangles in $\A$ such that $f, u$ are $\X$-monics in $\C$
\[\xy\xymatrixcolsep{2pc}\xymatrix@C14pt@R14pt{
A\ar[r]^f\ar[d]_\alpha&B\ar[r]^g\ar[d]^\beta&C\ar[r]^-h\ar[d]^\gamma& \Sigma(A)\ar[d]^{\Sigma(\alpha)}\\
L\ar[r]^u&M\ar[r]^v&N\ar[r]^-w&\Sigma(L)}
\endxy\]
Then $\Sigma^\X(\ul{\alpha})\circ \ul{\xi}(f, g)=\ul{\xi}(u, v)\circ \ul{\gamma} $ in $\C/\X$.
\end{lemma}

\begin{proof} \ By (\ref{kappaf}) and (\ref{xif}), we have the following two commutative diagrams:
\[\xy\xymatrixcolsep{2pc}\xymatrix@C14pt@R14pt{
A\ar[r]^f\ar@{=}[d]&B\ar[r]^g\ar[d]^{\delta^f}&C\ar[r]^-h\ar[d]^{\xi(f,g)}& \Sigma(A)\ar@{=}[d] & & A\ar[r]^f\ar[d]_\alpha&B\ar[r]^g\ar[d]^\beta&C\ar[r]^-h\ar[d]^\gamma& \Sigma(A)\ar[d]^{\Sigma(\alpha)}\\
A\ar[r]^{i^A}\ar[d]_\alpha&X^A\ar[r]^{p^A}\ar[d]^{\sigma^\alpha}&U^A\ar[r]^-{q^A}\ar[d]^{\kappa^\alpha}&\Sigma(A)\ar[d]^{\Sigma(\alpha)} & \mbox{and}& L\ar[r]^u\ar@{=}[d]&M\ar[r]^v\ar[d]^{\delta^u}&N\ar[r]^-w\ar[d]^{\xi(u,v)}&\Sigma(L)\ar@{=}[d]\\
L\ar[r]^{i^L}&X^L\ar[r]^{p^L}&U^L\ar[r]^-{q^L}&\Sigma(L) & & L\ar[r]^{i^L}&X^L\ar[r]^{p^L}&U^L\ar[r]^-{q^L}&\Sigma(L)}
\endxy\]
Thus the assertion follows from Lemma \ref{lemma:notation} directly. \end{proof}

Given any morphism $f\colon A\to B$ in $\C$, by (PRT1) (ii) and (iii), there are right $\C$-triangles $A\xrightarrow{\left(\begin{smallmatrix}
		i^A \\
		f
		\end{smallmatrix}\right)}  X^A\oplus B\xrightarrow{(\theta, -\eta)} N\stackrel{\lambda}\to \Sigma(A)$ and  $A\xrightarrow{\left(\begin{smallmatrix}
	1 \\
	f
	\end{smallmatrix}\right)} A\oplus B \xrightarrow{(f, -1)} B\stackrel{0}\to \Sigma(A)$.  Note that $A \oplus B \xrightarrow{\left(\begin{smallmatrix}
	i^A & 0 \\
	0 & 1
	\end{smallmatrix}\right)} X^A\oplus B\xrightarrow{(p^A, 0)} U^A\xrightarrow{\left(\begin{smallmatrix}
	q^A \\
	0
	\end{smallmatrix}\right)}\Sigma(A) \oplus \Sigma(B)$ is also a right $\C$-triangle since $\R(\C)$ is closed under finite direct sums. We have the following commutative diagram by (PRT4):
\begin{equation}\label{pullback}
\xy\xymatrixcolsep{2pc}\xymatrix@C28pt@R14pt{
A\ar[r]^{\left(\begin{smallmatrix}
	1 \\
	f
	\end{smallmatrix}\right)}\ar@{=}[d]& A\oplus B\ar[r]^-{(f, -1)}\ar[d]^{\left(\begin{smallmatrix}
	i^A & 0 \\
	0 & 1
	\end{smallmatrix}\right)}& B\ar[r]^0\ar[d]^\eta & \Sigma(A)\ar@{=}[d]\\
A\ar[r]^{\left(\begin{smallmatrix}
	i^A \\
	f
	\end{smallmatrix}\right)}& X^A\oplus B\ar[r]^-{(\theta,-\eta)}\ar[d]^{(p^A, 0)} & N\ar[r]^\lambda\ar[d]^\zeta & \Sigma(A)\ar[d]^{\left(\begin{smallmatrix}
	1 \\
	\Sigma(f)
	\end{smallmatrix}\right)}\\
& U^A\ar@{=}[r]\ar[d]^{\left(\begin{smallmatrix}
	q^A \\
	0
	\end{smallmatrix}\right)} & U^A\ar[r]^-{\left(\begin{smallmatrix}
	q^A \\
	0
	\end{smallmatrix}\right)}\ar[d]& \Sigma(A)\oplus\Sigma(B)\\
& \Sigma(A)\oplus \Sigma(B)\ar[r]^-{(\Sigma(f),-1)}& \Sigma(B)}
\endxy
\end{equation}	
such that the second column from the right is a right $\C$-triangle and $\eta$ is an $\X$-monic. By Lemma \ref{lem:octahedral}, we have $\ul{\xi}(\eta, \zeta)=\ul{\kappa}^{(f, -1)}\circ \ul{\xi}(\left(\begin{smallmatrix}
	i^A & 0 \\
	0 & 1
	\end{smallmatrix}\right), (p^A, 0))$. It can be proved that the later composition is just $\ul{\kappa}^f$ in $\C/\X$ by the proof of Lemma \ref{lemma:uniqueness} (ii) and the constructions of $\kappa^{(f, -1)}$ and $\xi(\left(\begin{smallmatrix}
	i^A & 0 \\
	0 & 1
	\end{smallmatrix}\right),(p^A, 0)) $.

\begin{lemma} \label{lem:dis-induce} \ $({\mathrm i)}$  \ $A\xrto{\ul{f}} B\xrto{-\ul{\eta}}N \xrto{\ul{\zeta}} \Sigma^\X(A)$ is a standard right triangle.

\ $(\mathrm{ii})$  \ $B\xrto{\ul{\eta}} N\xrto{\ul{\zeta}} \Sigma^\X(A)\xrto{\Sigma^\X(\ul{f})}\Sigma^\X(B)$ is a standard right triangle.

$(\mathrm{iii})$ \ If $A\xrto{f} B\xrto{g} C\xrto{h} \Sigma(A)$ is a right $\C$-triangle with $f$ an $\X$-monic in $\C$, then the standard right triangle $A\xrto{\ul{f}} B \xrto{\ul{g}} C \xrightarrow{\ul{\xi}(f, g)}\Sigma^\X(A)$ is isomorphic to the standard right triangle $A\xrto{\ul{f}} B\xrto{-\ul{\eta}} N \xrto{\ul{\zeta}} \Sigma^\X(A)$.
\end{lemma}
\begin{proof} \ (i) \ By (\ref{pullback}), there is a commutative diagram of right $\C$-triangles:
		\[\xy\xymatrixcolsep{2pc}\xymatrix@C20pt@R14pt{
A\ar[r]^-{\left(\begin{smallmatrix}
	i^A \\
	f
	\end{smallmatrix}\right)}\ar@{=}[d]&X^A\oplus B\ar[r]^-{(\theta,-\eta)}\ar[d]^{(1,0)}&N\ar[r]^-\lambda\ar[d]^\zeta& \Sigma(A)\ar@{=}[d]\\
A\ar[r]^-{i^A}&X^A\ar[r]^-{p^A}&U^A\ar[r]^-{q^A}&\Sigma(A)}
\endxy\]
	Thus $\ul{\zeta}=\ul{\xi}(\left(\begin{smallmatrix}
			i^A \\
			f
			\end{smallmatrix}\right), (\theta, -\eta))$ by Lemma \ref{lemma:notation} and $A\xrto{\ul{f}} B\xrightarrow{-\ul{\eta}} N \xrto{\ul{\zeta}} \Sigma^\X(A)$ is a standard right triangle.
	
(ii) \ This follows from (\ref{pullback}) and $\ul{\xi}(\eta, \zeta)=\ul{\kappa}^f$.

	(iii)  Consider the following diagram of right $\C$-triangles:
	\[\xy\xymatrixcolsep{2pc}\xymatrix@C20pt@R16pt{
A\ar[r]^-f\ar@{=}[d]&B\ar[r]^-g\ar[d]^{\left(\begin{smallmatrix}
			\delta^f \\
			1
			\end{smallmatrix}\right)}&C\ar[r]^-h\ar@{.>}[d]^t& \Sigma(A)\ar@{=}[d]\\
A\ar[r]^-{\left(\begin{smallmatrix}
			i^A \\
			f
			\end{smallmatrix}\right)}&X^A\oplus B\ar[r]^-{(\theta, -\eta)}&N\ar[r]^\lambda&\Sigma(A)}
\endxy\]
	where $\delta^f\colon B\to X^A$ is constructed in (\ref{xif}) which satisfies $\delta^f\circ f=i^A$. Thus there exists a morphism $t\colon C\to N$ making the above diagram commutative by (PRT3). Then $\ul{\xi}(f, g)=\ul{\zeta}\circ \ul{t} $ %$\ul{\xi}(\left(\begin{smallmatrix}	i^A \\	f	\end{smallmatrix}\right), (\theta, -\eta))\circ\ul{t} =$
in $\C/\X$ by Lemma \ref{lem:octahedral}. So we have a commutative diagram of standard right triangles in $\C/\X$:
\[\xy\xymatrixcolsep{2pc}\xymatrix@C20pt@R14pt{
A\ar[r]^-{\ul{f}}\ar@{=}[d]&B\ar[r]^-{\ul{g}}\ar@{=}[d]&C\ar[r]^-{\ul{\xi}(f,g)}\ar[d]^{\ul{t}}& \Sigma^\X(A)\ar@{=}[d]\\
A\ar[r]^-{\ul{f}}& B\ar[r]^-{-\ul{\eta}}&N\ar[r]^-{\ul{\zeta}}&\Sigma^\X(A)}
\endxy\]

 We will show that $\ul{t}$ is an isomorphism. In fact, by (PRT3), there is a morphism $\tau\colon N\to C$ such that the following diagram of right $\C$-triangles is commutative
	\[\xy\xymatrixcolsep{2pc}\xymatrix@C20pt@R14pt{
A\ar[r]^-{\left(\begin{smallmatrix}
	i^A \\
	f
	\end{smallmatrix}\right)}\arrow@{=}[d]&X^A\oplus B\ar[r]^-{(\theta,-\eta)}\ar[d]^{(0,1)}&N\ar[r]^-\lambda\ar[d]^\tau& \Sigma(A)\ar@{=}[d]\\
A\ar[r]^-f&B\ar[r]^-g&C\ar[r]^-h&\Sigma(A)}
\endxy\]
	It can be shown that $\tau\circ t$ is an isomorphism by the proof of Lemma \ref{lemma:uniqueness} (i). Since $(1, -\delta^f)\circ\left(\begin{smallmatrix}
	i^A\\
	f
	\end{smallmatrix}\right)=i^A-\delta^f\circ f=0$ by (\ref{xif}), there is a morphism $l\colon N\to X^A$ such that $l\circ (\theta, -\eta)=(1, -\delta^f)$ since $(\theta, \eta)$ is a weak cokernel of $\left(\begin{smallmatrix}
	i^A\\
	f
	\end{smallmatrix}\right)$. Then we have the following commutative diagram of right $\C$-triangles
\[\xy\xymatrixcolsep{2pc}\xymatrix@C20pt@R14pt{
A\ar[r]^-{\left(\begin{smallmatrix}
	i^A \\
	f
	\end{smallmatrix}\right)}\ar@{=}[d]&X^A\oplus B\ar[r]^-{(\theta,-\eta)}\ar@{=}[d]&N\ar[r]^-\lambda\ar[d]^{t\circ \tau+\theta\circ l}& \Sigma(A)\ar@{=}[d]\\
A\ar[r]^-{\left(\begin{smallmatrix}
	i^A \\
	f
	\end{smallmatrix}\right)}&X^A\oplus B\ar[r]^-{(\theta, -\eta)}&N\ar[r]^{\lambda}&\Sigma(A)}
\endxy\]
 So $t\circ \tau+ \theta \circ l$ is an isomorphism by the proof of Lemma \ref{lemma:uniqueness} (i), and thus $\ul{t}\circ \ul{\tau}$ is an isomorphism. Therefore $\ul{t}$ is an isomorphism and we are done.\end{proof}

\subsection*{The proof of Theorem \ref{thm:main}}
\begin{proof}  (i) \ We verify (RT1)-(RT4) of Definition \ref{def:rtricat} one by one.
	
	(RT1) \  For each object $A\in \C$, there is a standard right triangle $0\to A\stackrel{1}\to A\to 0$ induced by the right $\C$-triangle $0\to A\xrto{1} A\to 0$. Given any morphism $f \colon A\to B$ in $\C$, by (PRT2) (iii), there is a right $\C$-triangle $A\xrightarrow{\left(\begin{smallmatrix}
		i^A \\
		f
		\end{smallmatrix}\right)}  X^A\oplus B\xrightarrow{(\theta, -\eta)} N \to \Sigma(A)$ which induces a standard right triangle $A\xrto{\ul{f}} B\xrightarrow{-\ul{\eta}} N\xrto{\ul{\zeta}} \Sigma^\X(A)$ by Lemma \ref{lem:dis-induce} (i).
	
	(RT2) \  Let $A\xrto{\ul{f}} B\xrightarrow{-\ul{\eta}} N\xrto{\ul{\zeta}} \Sigma^\X(A)$ be a standard right triangle induced by the right $\C$-triangle $A\xrightarrow{\left(\begin{smallmatrix}
		i^A \\
		f
		\end{smallmatrix}\right)}  X^A\oplus B\xrightarrow{(\theta, -\eta)} N \to \Sigma(A)$. Then $B\xrightarrow{-\ul{\eta}} N\xrto{\ul{\zeta}} \Sigma^\X(A)\xrightarrow{-\Sigma^\X(\ul{f})}\Sigma^\X(B)$ is a distinguished right triangle since it is isomorphic to the standard right triangle $B\xrto{\ul{\eta}} N\xrto{\ul{\zeta}} \Sigma^\X(A)\xrightarrow{\Sigma^\X(\ul{f})}\Sigma^\X(B)$ (see Lemma \ref{lem:dis-induce} (ii)).

(RT3) \  Assume that we have a diagram of standard right triangles in $\C/\X$:
\[\xy\xymatrixcolsep{2pc}\xymatrix@C16pt@R14pt{
A\ar[r]^-{\ul{f}}\ar[d]_{\ul{\alpha}}&B\ar[r]^-{\ul{g}}\ar[d]^{\ul{\beta}}&C\ar[r]^-{\ul{\xi}(f,g)}& \Sigma^\X(A)\ar[d]^{\Sigma^\X(\ul{\alpha})}\\
L\ar[r]^-{\ul{u}}&M\ar[r]^-{\ul{v}}&N\ar[r]^-{\ul{\xi}(u,v)}&\Sigma^\X(L)}
\endxy\]
with the leftmost square commutative. Since  $\ul{u} \circ \ul{\alpha}=\ul{\beta} \circ \ul{f}$, there is a morphism $s\colon X^A\to M$ such that $u \circ \alpha-\beta \circ f=s\circ i^A$. By (\ref{xif}), there is a morphism $\delta^f \colon B\to X^A$ such that $i^A= \delta^f \circ f$ and thus $(\beta+s\circ \delta^f)\circ f=u \circ \alpha$. So by (PRT3), there is a morphism $t\colon C\to N$ making the following diagram of right $\C$-triangles commutative
		\[\xy\xymatrixcolsep{2pc}\xymatrix@C20pt@R14pt{
A\ar[r]^-f\ar[d]_\alpha&B\ar[r]^-g\ar[d]^{\beta+s\circ \delta^f}&C\ar[r]\ar@{.>}[d]^t& \Sigma(A)\ar[d]^{\Sigma(\alpha)}\\
L\ar[r]^u&M\ar[r]^v&N\ar[r]&\Sigma(L)}
\endxy\]	
	Then $\ul{\xi}(u, v)\circ \ul{t}=\Sigma^\X(\ul{\alpha})\circ \ul{\xi}(f, g)$ by Lemma \ref{lem:octahedral}. Therefore $\ul{t}: C\to N$ is the desired filler.

(RT4). Assume that we have three standard right triangles $A\xrto{\ul{f}} B\xrto{\ul{l}} C'\xrightarrow{\ul{\xi}(f, l)}\Sigma^\X(A)$, $B\xrto{\ul{g}} C\xrto{\ul{h}} A'\xrightarrow{\ul{\xi}(g, h)} \Sigma^\X(B)$ and $A\xrightarrow{\ul{g\circ f}} C\xrto{\ul{m}} B'\xrightarrow{\ul{\xi}(g\circ f, m)}\Sigma^\X(A)$. Then we have a commutative diagram in $\A$ by (PRT4):
	\begin{equation}
\label{octa}
	\xy\xymatrixcolsep{2pc}\xymatrix@C14pt@R14pt{
A\ar[r]^f\ar@{=}[d]&B\ar[r]^l\ar[d]^g&C'\ar[r]\ar[d]^r& \Sigma(A)\ar@{=}[d]\\
A\ar[r]^-{g\circ f}&C\ar[r]^k\ar[d]^h&B'\ar[r]^m\ar[d]^s&\Sigma(A)\ar[d]^{\Sigma(f)}\\
&A'\ar@{=}[r]\ar[d]^j&A'\ar[r]^j\ar[d]&\Sigma(B)\\
&\Sigma(B)\ar[r]^{\Sigma(l)}&\Sigma(C')}
\endxy
	\end{equation}
such that the second column from the right is a right $\C$-triangle with $r$ an $\X$-monic. The commutative diagram (\ref{octa}) induces a diagram of right triangles in $\Delta^\X$:
	\[\xy\xymatrixcolsep{2pc}\xymatrix@C16pt@R14pt{
A\ar[r]^-{\ul{f}}\ar@{=}[d]&B\ar[r]^-{\ul{l}}\ar[d]^{\ul{g}}&C'\ar[r]^-{\ul{\xi}(f,l)}\ar[d]^{\ul{r}}& \Sigma^\X(A)\ar@{=}[d]\\
A\ar[r]^-{\ul{g\circ f}}&C\ar[r]^-{\ul{k}}\ar[d]^{\ul{h}}&B'\ar[r]^-{\ul{\xi}(g\circ f,k)}\ar[d]^{\ul{s}}&\Sigma^\X(A)\ar[d]^{\Sigma^\X(\ul{f})}\\
&A'\ar@{=}[r]\ar[d]^{\ul{\xi}(g,h)}&A'\ar[r]^-{\ul{\xi}(g,h)}\ar[d]^{\ul{\xi}(r,s)}&\Sigma^\X(B)\\
&\Sigma^\X(B)\ar[r]^{\Sigma^\X(\ul{l})}&\Sigma^\X(C')}
\endxy\]
To finish the proof of (RT4), we have to show that the lowest square and the rightmost two squares are commutative. But these can be obtained from Lemma \ref{lem:octahedral} by the commutative diagram (\ref{octa}) and the following commutative diagram:
\[\xy\xymatrixcolsep{2pc}\xymatrix@C14pt@R14pt{
A\ar[r]^-{g\circ f}\ar[d]_f&C\ar[r]^-k\ar@{=}[d]&B'\ar[r]^-m\ar[d]^s& \Sigma(A)\ar[d]^{\Sigma(f)}\\
B\arrow[r]^-g&C\ar[r]^h&A'\ar[r]^j&\Sigma(B)}
\endxy\]	
	The statement (ii) can be proved dually. \end{proof}

\section{Examples of partial one-sided triangulated categories}
In this section we construct examples of partial one-sided triangulated categories from cotorsion pairs in exact categories, torsion pairs and mutation pairs in triangulated categories.

\subsection*{Partial one-sided triangulated categories from additive categories}

 Let $\C$ be an additive category and $\R(\C)$ be the class of right exact sequences $ A\xrto{f} B\to C\to 0$ in $\C$ (here the notion of right exactness has its usual meaning).

\begin{proposition} \label{prop:pltcadditive} Let $\X$ be an additive subcategory of $\C$. Assume the following conditions hold:
	
	$(\mathrm a)$ \ For each $A\in \C$, there is a sequence $ A\stackrel{i}\to X\to U\to 0$ in $\R(\C)$ with $i$ an $\X$-preenvelope.
	
	$(\mathrm b)$ \ If $ A\stackrel{i}\to X\to U\to 0$ is in $\R(\C)$ with $i$ an $\X$-preenvelope, then for any morphism $f\colon A\to B$, there is a sequence $A\xrightarrow{\left(\begin{smallmatrix}
		i \\
		f
		\end{smallmatrix}\right)} X\oplus B\to N\to 0$ in $\R(\C)$.
	
	\noindent Then $(\C, 0, \R(\C), \X)$ is a partial right triangulated category.
\end{proposition}

\begin{proof} \ By the construction of $\R(\C)$, it is closed under isomorphisms and finite direct sums.
		
	(PRT1) \  The statements (i) and (iii) follow from the conditions (a) and (b).	

	For (ii), let $f\colon A\to B$ be a morphism. Then $ A \xrto{\left(\begin{smallmatrix}
		1 \\
		f
		\end{smallmatrix}\right)} A\oplus B\xrto{(f, -1)} B\to 0$ is a split right exact sequence and thus in $\R(\C)$.

(PRT2) \  Assume that we have a commutative diagram of right exact sequences in $\C$
	\[\xy\xymatrixcolsep{2pc}\xymatrix@C14pt@R14pt{
A\ar[r]^f\ar[d]_\alpha&B\ar[r]^g\ar[d]^\beta&C\ar[r]\ar[d]^\gamma& 0\\
A'\ar[r]^i&X\ar[r]^p&U\ar[r]&0}
\endxy\]	
such that $X\in \X$. If there is a morphism $s\colon B\to A'$ such that $\alpha=s\circ f$, then $(\beta-i\circ s)\circ f=\beta\circ f-i\circ s \circ f=\beta\circ f-i\circ \alpha=0$. Thus there is a morphism $t\colon C\to X$ such that $\beta-i\circ s=t\circ g$ since $g$ is a cokernel of $f$. Therefore $p\circ t\circ g=p\circ (i\circ s+t\circ g)=p\circ \beta=\gamma \circ g$. Then $\gamma=p\circ t$ since $g$ is an epimorphism.

	(PRT3) follows from the universal property of cokernels.

	(PRT4) Let $A\xrto{f} B\xrto{l} C'\to 0$, $B\xrto{g} C \xrto{h} A'\xrto{j} 0$ and $A\xrightarrow{g\circ f} C \xrto{k} B'\to 0$ be three sequences in $\R(\C)$ such that $f, g$ are $\X$-monics. Then there are morphisms $r$ and $s$ by the universal property of cokernels making the following diagram commutative:
	\[\xy\xymatrixcolsep{2pc}\xymatrix@C14pt@R14pt{
A\ar[r]^f\ar@{=}[d]&B\ar[r]^l\ar[d]^g&C'\ar[r]\ar[d]^r& 0\\
A\ar[r]^{g\circ f}&C\ar[r]^k\ar[d]^{h}&B'\ar[r]\ar[d]^s&0\\
&A'\ar@{=}[r]\ar[d]&A'\ar[r]\ar[d]&0\\
&0 &0}
\endxy\]
It can be verified that the upper right square is a pushout diagram, and then the rightmost column is a right exact sequence with $r$ an $\X$-monic. \end{proof}

Dually, if $\C$ is an additive category and $\rmL(\C)$ the class of left exact sequences $ 0\to K\to A\xrto{f} B$ in $\C$. We have
\begin{proposition} \label{prop:dualpltcadditive} Let $\X$ be an additive subcategory of $\C$. Assume the following conditions hold:

	$(\mathrm{a'})$ \ For  each object $A\in \C$, there is a sequence $0\to K\to X\xrto{\pi} A$ in $\rmL(\C)$ with $\pi$ an $\X$-precover.
	
	$(\mathrm{b'})$ \ If $0\to K\to X\stackrel{\pi}\to B$ is in $\rmL(\C)$ with $\pi$ an $\X$-cover, then for any morphism $f\colon A\to B$, there is a sequence $0\to M\to A\oplus X\xrightarrow{(f, \pi)} B$ in $\rmL(\C)$.
	
	\noindent  Then $(\C, 0, \rmL(\C), \X)$ is a partial left triangulated category.
\end{proposition}

\subsection*{Partial one-sided triangulated categories from exact categories}
Let $\A$ be an additive category. A {\it kernel-cokernel sequence} in $\A$ is a sequence $A\stackrel{i}\to B\stackrel{d}\to C $
such that $i=\Ker d$ and $d=\Coker i$. Let $\mathcal{E}$ be a class of kernel-cokernel sequences of $\A$. Following Keller \cite[Appendix A]{Keller90}, we call a kernel-cokernel sequence a {\it conflation} if it is in $\mathcal{E}$. A morphism $i$ is called an {\it inflation} if there is a conflation $ A\stackrel{i}\to B\stackrel{d} \to C$ and the morphism $d$ is called a {\it deflation}.

Recall that an {\it exact structure} on an additive category $\A$ is a class $\E$ of kernel-cokernel sequences which is closed under isomorphisms and satisfies the following axioms due to Quillen \cite{Quillen73} and Keller \cite[Appendix A]{Keller90}:
\vskip5pt
$(\mathrm {Ex0})$ \ The identity morphism of the zero object is an inflation.

$(\mathrm{Ex1})$ \ The class of deflations is closed under composition.

$(\mathrm{Ex1})^{\mathrm{op}}$ \ The class of inflations is closed under composition.

$(\mathrm{Ex2})$ \ For any deflation $d\colon A\to B$ and any morphism $f\colon B'\to B$, there exists a pullback diagram such that $d'$ is a deflation:
\[\xy\xymatrixcolsep{2pc}\xymatrix@C14pt@R14pt{
A'\ar@{.>}[r]^{d'}\ar@{.>}[d]_{f'}&B'\ar[d]^f\\
A\ar[r]^d&B}
\endxy\]

$(\mathrm{Ex2})^{\mathrm{op}}$ \ For any inflation $i\colon C\to D$ and any morphism $g\colon C\to C'$, there is a pushout diagram such that $i'$ is an inflation:	
\[\xy\xymatrixcolsep{2pc}\xymatrix@C14pt@R14pt{
C\ar[r]^i\ar[d]_g&D\ar@{.>}[d]^{g'}\\
C'\ar@{.>}[r]^{i'}&D'}
\endxy\]	

An {\it exact category} is a pair $(\A, \E)$ consisting of an additive category $\A$ and an exact structure $\E$ on $\A$.  We refer the reader to \cite{Buhler10} for a readable introduction to exact categories. Sometimes we suppress the class $\mathcal{E}$ and just say that $\A$ is an exact category.

In an exact category $(\A, \mathcal{E})$, we can define the Yoneda Ext bifunctor $\Ext^1_\A(C,A)$. It is the abelian group of equivalence classes of conflations $ A\to  B\to C $ in $\mathcal{E}$; see \cite[Chapter XII.4]{MacLane} for details.

 Now let $(\A, \E)$ be an exact category and $\X$ an additive subcategory of $\A$.
 \begin{definition} \label{defn:specialclosedexact} \ An additive subcategory $\C$ of $\A$ is said to be {\it special $\X$-monic closed} if

 (a) \ $\X\subseteq\C$ and for each object $A\in \C$, there is a conflation $A\stackrel{i}\to X\to U$ with $U\in \C$ and $i$ an $\X$-preenvelope.

 (b) \ If $A\stackrel{i}\to X\to U$ is a conflation in $\C$ with $i$ an $\X$-preenvelope, then for any morphism $f\colon A\to B$ there is a conflation $A\xrightarrow{\left(\begin{smallmatrix}
	i \\
	f
	\end{smallmatrix}\right)}  X\oplus B \to N$ with $N\in \C$.
\end{definition}

Dually, we have the notion of a {\it special $\X$-epic closed} additive subcategory of $\A$.

\begin{proposition} \label{prop:pltcexact} \	Let $(\A,\mathcal{E})$ be an exact category and $\X,\C$ two additive subcategories of $\A$. Then

 $(\mathrm{i})$ \ If $\C$ is special $\X$-monic closed, then $(\A, 0, \R(\C), \X)$ is a partial right triangulated category, where $\R(\C)=\{ A\xrto{f} B\to C \to 0 \ | \  A\xrto{f} B\to C \in \E,  C\in \C \}$.

$(\mathrm{ii})$ \ If $\C$ is special $\X$-epic closed, then $(\A, 0, \rmL(\C), \X)$ is a partial left triangulated category, where $\rmL(\C)=\{ 0\to K\to A\xrto{f}B \ |  \ K\to A\xrto{f} B \in \E, K\in \C \}$.
\end{proposition}

\begin{proof}	 Similar to the proof of Proposition \ref{prop:pltcadditive}, we can prove the statement (i) by noting \cite[Proposition 2.12]{Buhler10}.
% (i)\  Since $\E$ is closed under isomorphisms and finite direct sums, then so is $\R(\C)$. With the same argument as the proof of Proposition \ref{prop:pltcadditive}, we know that $\R(\A, \X)$ satisfies (PRT1) (i)$-$(ii), (PRT2) and (PRT3). For (PRT1) (iii).  Let $A\stackrel{i}\to X\to U$ be in $\R(\A, \X)$ with $i$ an $\X$-monic inflation. Then for any morphism $f\colon A\to B$, by $(\mathrm{EX}2)^{\mathrm{op}}$, we have a pushout diagram
	%\[\xy\xymatrixcolsep{2pc}\xymatrix@C16pt@R16pt{
%A\ar[r]^i\ar[d]_{f}&X\ar[d]^\theta\\
%B\ar[r]^\eta&N}
%\endxy\]
%	Thus $A\xrto{\left(\begin{smallmatrix}
%		i \\
%		f
%		\end{smallmatrix}\right)} X \oplus B\xrto{(\theta, -\eta)} N\to 0$ is in $\R(\A, \X)$ by \cite[Proposition 2.12]{Buhler10}. Since
 %inflations are closed under pushout, the proof of (PRT4) is the same as the one in Proposition \ref{prop:pltcadditive}.
 The statement (ii) can be proved dually.  \end{proof}

\subsection*{Partial one-sided triangulated categories from cotorsion pairs}

\begin{definition} \ \cite[Definition 2.1]{Gillespie11} Let $\A$ be an exact category. A {\it cotorsion pair} in $\A$ is a pair $(\C, \F)$ of classes of objects of $\A$ such that $\C=\{C\in \A \ | \ \Ext^1_\A(C, F)=0, \forall \ F \in \F\}$ and  $\F=\{F\in \A \ | \ \Ext^1_\A(C, F)=0, \forall \ C\in \C \}.$
\end{definition}
\noindent The cotorsion pair $(\C, \F)$ is called {\it complete} if it has {\it enough projectives}, i.e. for each $A\in \A$ there is a conflation $F\to  C\to A$ such that $C\in \C, F\in \F$, and {\it enough injectives}, i.e. for each $A\in \A$, there is a conflation $A\to F'\to C'$ such that $C'\in \C, F'\in \F$.

If $(\C, \F)$ is a cotorsion pair in an exact category $(\A, \E)$, then both $\C$ and $\F$ are exact subcategories of $\A$ with the induced exact structures by $\E$. Thus by Proposition \ref{prop:pltcexact}, we have

\begin{corollary}\label{cotorsion}\	Let $(\C, \F)$ be a cotorsion pair in an exact category $(\A, \mathcal{E})$. Then

$(\mathrm{i})$ \ If $(\C, \F)$ has enough projectives, then both $(\A, 0, \rmL(\A), \C)$ and $(\F, 0, \rmL(\F), \C\cap \F)$ are partial left triangulated categories.

	$(\mathrm{ii})$ \ If $(\C,\F)$ has enough injectives, then both $(\A, 0, \R(\A), \F)$ and $(\C, 0, \R(\C), \C\cap \F)$ are partial right triangulated categories.	
\end{corollary}
\subsection*{Partial one-sided triangulated categories from one-sided triangulated categories}

 Let $(\T, \Sigma, \Delta)$ be a right triangulated category. Then for any right triangle $A\stackrel{f}\to B\stackrel{g}\to C\stackrel{h} \to \Sigma(A)$, the morphism $g$ is a weak cokernel of $f$ and $h$ is a weak cokernel of $g$ by \cite[Lemma 1.3]{ABM}. Let $\X$ be an additive subcategory of $\T$. An additive subcategory $\C$ in $\T$ is said to be {\it special $\X$-monic closed} if the following conditions are satisfied:

 (a) \ $\X\subseteq \C$ and for each $A\in \C$, there is a right
 triangle $A\stackrel{i}\to X\to U\to \Sigma(A)$ in $\Delta$ with $U\in \C$ and $i$ an $\X$-preenvelope.

 (b) \ Assume that $A\stackrel{i}\to X\to U\to \Sigma(A)$ is in $\Delta$ with $U\in \C$ and $i$ an $\X$-preenvelope in $\C$. Then for any morphism $f\colon A\to B$ in $\C$, there is a right triangle $A\xrightarrow{\left(\begin{smallmatrix}
	i \\
	f
	\end{smallmatrix}\right)}  X\oplus B \to N \to \Sigma(A)$ in $\Delta$ with $N\in \C$.

 (c) \ If $A\xrto{i} X\xrto{p} U\xrto{q} \Sigma(A)$ is in $\Delta$ with $U\in \C$ and $X\in\X$, then $p$ is a weak kernel of $q$.

\begin{proposition} \label{prop:prtc}\   Let $(\T, \Sigma, \Delta)$ be a right triangulated category. Let $\X\subseteq \C$ be two additive subcategories of $\T$. If $\C$ is special $\X$-monic closed, then $(\T, \Sigma, \R(\C), \X)$ is a partial right triangulated category, where $\R(\C)=\{A \xrto{f} B\to C\to \Sigma(A) \in  \Delta \ | \ C\in \C \}.$
 \end{proposition}
\begin{proof} \   Since $\Delta$ is closed under isomorphisms and finite direct sums, then so is $\R(\C)$.

(PRT1) \ The statements (i) and (iii) follow from the definition of a special $\X$-monic closed subcategory. For (ii), let $f\colon A\to B$ be a morphism in $\C$. Then it can be proved directly that the sequence $A\xrightarrow{\left(\begin{smallmatrix}
	1 \\
	f
	\end{smallmatrix}\right)} A\oplus B\xrightarrow{(f, -1)}B\xrto{0} \Sigma(A)$ is in $\Delta$, and thus it is in $\R(\C)$.

(PRT2) \ Assume that we have a commutative diagram
	\[\xy\xymatrixcolsep{2pc}\xymatrix@C14pt@R14pt{
A\ar[r]^f\ar[d]_\alpha&B\ar[r]^g\ar[d]^\beta&C\ar[r]^-h\ar[d]^\gamma& \Sigma(A)\ar[d]^{\Sigma(\alpha)}\\
A'\ar[r]^i&X\ar[r]^p&U\ar[r]^-q&\Sigma(A')}
\endxy\]
with rows in $\R(\C)$ and $X\in \X$. If there is a morphism $s\colon B\to A'$ such that  $\alpha=s\circ f$, then $q\circ  \gamma=\Sigma(\alpha)\circ h= \Sigma(s)\circ \Sigma(f)\circ h=0$ by \cite[ Lemma 1.3]{ABM}. Thus there is a morphism $t\colon C\to X$ such that $\gamma=p\circ t$ since $p$ is a weak kernel of $q$ by assumption.

(PRT3) \  This follows from (RT3) directly.

(PRT4) \ Let $A\xrto{f}B\xrto{l} C'\to \Sigma(A)$, $B\xrto{g} C \xrto{h} A'\xrto{j} \Sigma(B)$ and $A\xrto{g\circ f} C \xrto{k} B'\to \Sigma(A)$ be three right triangles in $\R(\C)$ such that $f, g$ are $\X$-monics. Then there is a commutative diagram
	\begin{equation*}
	\xy\xymatrixcolsep{2pc}\xymatrix@C14pt@R14pt{
A\ar[r]^f\ar@{=}[d]&B\ar[r]^l\ar[d]^g&C'\ar[r]\ar[d]^r& \Sigma(A)\ar@{=}[d]\\
A\ar[r]^-{g\circ f}&C\ar[r]^-k\ar[d]^h&B'\ar[r]\ar[d]^s&\Sigma(A)\ar[d]^{\Sigma(f)}\\
&A'\ar@{=}[r]\ar[d]^j&A'\ar[r]^-j\ar[d]&\Sigma(B)\\
&\Sigma(B)\ar[r]^{\Sigma(l)}&\Sigma(C')}
\endxy
	\end{equation*}
whose second column from the right is in $\Delta$ by (RT4). By \cite[Corollary 1.4 (a)]{ABM}, the above commutative diagram induces a commutative diagram with exact rows
\[\xy\xymatrixcolsep{2pc}\xymatrix@C14pt@R14pt{
\Hom_\T(\Sigma(A),\X)\ar[r]\ar@{=}[d]&\Hom_\T(B',\X)\ar[r]\ar[d]^{r^*}&\Hom_\T(C,\X)\ar[r]\ar[d]^{g^*}& \Hom_\T(A,\X)\ar@{=}[d]\\
\Hom_\T(\Sigma(A),\X)\ar[r]&\Hom_\T(C',\X)\ar[r]&\Hom_\T(B,\X)\ar[r]&\Hom_\T(A,\X)}
\endxy\]
Since $g^*$ is an epimorphism, by the weak four lemma we know that $r^*$ is also an epimorphism. Thus $r$ is an $\X$-monic. So $C'\xrto{r}B'\xrto{s}A'\xrto{\Sigma(l)\circ j} \Sigma(C')$ is in $\R(\C)$.
\end{proof}

Dually, let $(\T, \Omega, \nabla)$ be a left triangulated category and $\X$ an additive subcategory of $\T$, we also have the notion of a {\it special $\X$-epic closed} subcategory of $\T$ and the following result:
\begin{proposition} \label{prop:pltc}\  Let $(\T,\Omega,\nabla)$ be a left triangulated category and $\X, \C$ two additive subcategories of $\T$. If $\C$ is special $\X$-epic closed, then $(\T, \Omega, \rmL(\C), \X)$ is a partial left triangulated category, where $\rmL(\C)=\{\Omega(B)\to K\to A \xrto{f} B \in  \nabla \ | \ K\in \C \}.$
 \end{proposition}

Let $(\T, [1], \Delta)$ be a triangulated category and $\C$ an additive subcategory of $\T$. Recall that $\C$ is said to be {\it extension-closed} if $A\to B\to C\to A[1]$ is a triangle in $\Delta$ such that $A, C\in \C$, then $B\in \C$. We have the following lemma.

\begin{lemma} \label{lem:specialclosed} \ Let $(\T, [1], \Delta)$ be a triangulated category and $\C$ an extension-closed additive subcategory of $\T$. Assume that $\X$ is an additive subcategory of $\T$ contained in $\C$. Then

$(\mathrm{i})$ \ If for each $A\in \C$, there is a triangle $A\stackrel{i}\to X\to U\to A[1]$ in $\Delta$ with $U\in \C$ and $i$ an $\X$-preenvelope, then $\C$ is special $\X$-monic closed.

$(\mathrm{ii})$ \ If for each $A\in \C$, there is a triangle $A[-1]\to K\to X\xrto{\pi} A$ in $\Delta$ with $K\in \C$ and $\pi$ an $\X$-precover, then $\C$ is special $\X$-epic closed.
\end{lemma}

\begin{proof} We only prove (i), the statement (ii) can be prove dually. By assumption, the condition (a) of a special $\X$-monic closed subcategory holds. Since $\T$ is a triangulated category, thus for any triangle $A\xrto{f}B\xrto{g}C\xrto{h}A[1]$ in $\Delta$, $g$ is a weak kernel of $h$ and then the condition (c) of a special $\X$-monic closed subcategory holds. For the condition (b), let $A\stackrel{i}\to X\to U\to \Sigma(A)$ be a triangle in $\Delta$ with $U\in\C$ and $i$ an $\X$-preenvelope in $\C$. If $f\colon A\to B$ is a morphism in $\C$, then there is a commutative diagram of triangles in $\Delta$:
\[\xy\xymatrixcolsep{2pc}\xymatrix@C14pt@R14pt{
A\ar[r]^i\ar[d]_f&X\ar[r]\ar[d]^{-g}&U\ar[r]\ar@{=}[d]& \Sigma(A)\ar[d]^{\Sigma(f)}\\
B\ar[r]^h&N\ar[r]&U\ar[r]&\Sigma(A')}
\endxy\]
such that the leftmost square is a homotopy pushout, i.e. $A\xrto{\left(\begin{smallmatrix}
	i \\
	f
	\end{smallmatrix}\right)}X\oplus B\xrto{(-g, h)} N\to A[1]$ is a triangle in $\Delta$. Since $\C$ is extension-closed, we know that $N\in \C$. \end{proof}

\begin{example} \label{exam:tricat} (i) \ Let $(\T,[1], \Delta)$ be a triangulated category. Let $\X$ and $\C$ be additive subcategories of $\T$ closed under direct summands. Assume that $(\C, \C)$ forms an {\it $\X$-mutation pair} in the sense of \cite[Definition 2.5]{Iyama-Yoshino}:

$(\mathrm{a})$ \ $\C$ is {\it extension-closed}.

$(\mathrm{b})$ \ $\X\subseteq \C$ and $\Hom_\T(\X[-1], \C)=0=\Hom_\T(\C, \X[1])$.

$(\mathrm{c})$ \ For any $A\in \C$, there exist triangles $A[-1]\to K_A\to X_A\to A$ and $ A\to X^A\to K^A\to A[1]$ in $\bigtriangleup$ such that $
X_A, X^A\in\X$ and $K_A, K^A\in \C$.

\noindent Then $\C$ is both special $\X$-epic and special $\X$-monic closed by Lemma \ref{lem:specialclosed}. Therefore $(\T, [1], \R(\C), \X)$ is a partial right triangulated category and $(\T, [-1], \rmL(\C), \X)$ is a partial left triangulated category, where the classes $\R(\C)$ and $\rmL(\C)$ are as constructed in Propositions \ref{prop:prtc} and \ref{prop:pltc}.

(ii) \ Let $(\T,[1], \Delta)$ be a triangulated category. Let $\X$ and $\C$ be additive subcategories of $\T$ closed under direct summands. If $\C$ is {\it $\X$-Frobenius} in the sense of \cite[Definition 3.2]{Beligiannis13}, then $(\T, [1], \R(\C), \X)$ is a partial right triangulated category and $(\T, [-1], \rmL(\C), \X)$ is a partial left triangulated category, where the classes $\R(\C)$ and $\rmL(\C)$ are as constructed in Propositions \ref{prop:prtc} and \ref{prop:pltc}.

(iii) \ Let $(\T, \Sigma, \Delta)$ be a right triangulated category and $\Y\subseteq \C$ two additive subcategories of $\T$ as in the setting of \cite{Liu-Zhu}, then $(\T, \Sigma, \R(\C), \Y)$ as constructed in Proposition \ref{prop:prtc} is a partial right triangulated category by \cite[Lemma 3.3]{Liu-Zhu} and the proof of \cite[Theorem 3.9]{Liu-Zhu}.
\end{example}
\subsection*{Partial one-sided triangulated categories from torsion pairs in triangulated categories}

Let $(\T, [1], \Delta)$ be a triangulated category. Recall that a pair $(\X, \Y)$ of additive subcategories of $\T$ is called a {\it torsion pair} (also called a {\it torsion theory} in \cite[Definition 2.2]{Iyama-Yoshino}) if $\Hom_\T(\X, \Y)=0$ and for each $T\in \T$, there is a triangle $X\to T\to Y\to X[1]$ in $\Delta$ with $X\in \X$ and $Y\in \Y$.

\begin{corollary}\label{cor:torsion} \  Let $(\X, \Y)$ be a torsion pair in the triangulated category $(\T, [1], \Delta)$. Then

$(\mathrm{i})$\ $(\T, [1], \Delta, \Y)$ is a partial right triangulated category.

 $(\mathrm{ii})$\ $(\T, [-1], \Delta, \X)$ is a partial left triangulated category.
\end{corollary}
\begin{proof} This is since $\T$ is both special $\Y$-monic closed and special $\X$-epic closed if $(\X,\Y)$ is a torsion pair in $\T$.
\end{proof}

\section{Exact and abelian subfactor categories}
In this section we use the results from the previous sections to construct exact subfactor categories from exact categories and abelian subfactor categories from triangulated categories.
\subsection*{Exact subfactor categories} Let $\A$ be an additive category. Recall that $\A$ is said to be {\it preabelian} if every morphism in $\A$ has a kernel and a cokernel.

Let $(\A, \E)$ be an exact category. If $\X$ is an additive subcategory of $\A$, then a conflation $A\stackrel{f}\to B\stackrel{g}\to C$ in $\E$ is said to be {\it $\X$-complete} if $f$ is an $\X$-monic and $g$ is an $\X$-epic. We have the following theorem which covers \cite[Theorem A (2)]{Kussin/Lenzing/Meltz} (compare \cite[Theorem 3.1]{Chen12} and \cite[Theorem 3.5]{Demonet/Iyama}):

\begin{theorem} \ \label{thm:sexact} Let $(\A, \E)$ be an exact category with $\X$ an additive subcategory. If for each object $A$ there are conflations $X_1\to X_0\stackrel{p}\to A $ and $ A\xrto{i} X^0\to X^1$ such that $p$ is an $\X$-precover, $i$ is an $\X$-preenvelope and $ X_1, X^1 \in \X$, then  $\A/\X$ is a preabelian category and has an exact structure induced by $\X$-complete conflations.
\end{theorem}
\begin{proof}\ By assumption and Proposition \ref{prop:pltcexact}, $\A$ has a partial right triangulated structure $(\R(\A), \X)$ with $\R(\A)=\{A\to B\to C\to 0 \ | A\to B\to C \in \E\}$, and a partial left triangulated structure $(\rmL(\A), \X)$ with $\rmL(\A)=\{  0\to K\to A\to B \ | \  K\to A\to B\in \E\}$. They induce a right triangulated structure $(\Sigma^\X, \Delta^\X)$ and a left triangulated structure $(\Omega_\X, \nabla_\X)$ on $\A/\X$ by Theorem \ref{thm:main}. By the constructions of  $\Sigma^\X$ and $\Omega_\X$, we know that they are zero functors. Therefore any morphism in $\A/\X$ has a kernel and a cokernel, i.e. $\A/\X$ is a preabelian category.

If $A\xrto{f} B\xrto{g} C$ is an $\X$-complete conflation, by the constructions of standard right and left triangles in $\A/\X$, we have a standard left triangle $0\to A\xrto{\ul{f}} B\xrto{\ul{g}} C$ and a standard right triangle $ A\xrto{\ul{f}} B\xrto{\ul{g}} C\to 0$. So $ A\xrto{\ul{f}} B\xrto{\ul{g}} C$ is a kernel-cokernel sequence in $\A/\X$.
	
 Let $\ul{\E}_\X$ be the class of kernel-cokernel sequences in $\A/\X$ induced by $\X$-complete conflations. We will show that $\ul{\E}_\X$ is an exact structure on $\A/\X$ by checking axioms of exact categories one by one. We only verify axioms $(\rm{Ex}0), (\rm{Ex}1)^{\rm op}$ and $(\rm{Ex2})^{\rm op}$ since the others can be proved dually.
	
	$(\mathrm{Ex}0)$ \ This is since $A\xrto{1} A\to 0$ is an $\X$-complete conflation for each $A\in \A$.
	
	$(\mathrm{Ex}1)^{\rm op}$ \ Let $\ul{f}\colon A\to B$ and $\ul{g}\colon B\to C$ be two composable inflations in $\A/\X$. Then there are two kernel-cokernel sequences $A\xrto{\ul{f}} B\xrto{\ul{l}} C' $ and $B\xrto{\ul{g}} C\xrto{\ul{h}} A'$ in $\mathcal{E}_\X$. Thus we have two $\X$-complete conflations $A \xrto{f} B \xrto{l} C'$ and $B\xrto{g} C \xrto{h} A'$ in $\A$. Since inflations in $\A$ are closed under composites, we have a conflation $A\xrightarrow{g\circ f}C\xrto{k} B'$ in $\A$. Similar to the proof of (PRT4) in Proposition \ref{prop:pltcexact}, we have a commutative diagram in $\A$
\[\xy\xymatrixcolsep{2pc}\xymatrix@C14pt@R14pt{
A\ar[r]^f\ar@{=}[d]&B\ar[r]^l\ar[d]^g&C'\ar[d]^r\\
A\ar[r]^-{g\circ f}&C\ar[r]^-k\ar[d]^-h&B'\ar[d]^s\\
&A'\ar@{=}[r]&A'}
\endxy\]
such that the rightmost column is a conflation. Since $h=s\circ k$ is an $\X$-epic, so is $s$. Thus the above diagram induces a commutative diagram with exact rows
\[\xy\xymatrixcolsep{2pc}\xymatrix@C14pt@R14pt{
0\ar[r]&\Hom_\A(\X,B)\ar[r]\ar[d]_{l_*}&\Hom_\A(\X,C)\ar[r]\ar[d]^{k_*}&\Hom_\A(\X,A')\arrow[r]\ar@{=}[d]& 0\\
0\ar[r]&\Hom_\A(\X,C')\ar[r]&\Hom_\A(\X, B')\ar[r]&\Hom_\A(\X, A')\ar[r]&0}
\endxy\]
Since the induced morphism $l_*$ is an epimorphism, so is $k_*$ by the Short Five Lemma. Thus $k$ is an $\X$-epic and the conflation $ A\xrightarrow{g\circ f} C\xrto{k} B'$ is $\X$-complete. So we have a sequence $ A\xrightarrow{\ul{g}\circ \ul{f}}C\xrto{\ul{k}} B' $ in $\ul{\E}_\X$. Therefore, $\ul{g}\circ \ul{f}$ is an inflation in $\A/\X$.
	
	$(\mathrm{Ex}2)^{\mathrm{op}}$ \ Let $\ul{f}\colon A\to B$ be an inflation in $\A/\X$ and $\ul{\alpha}\colon A\to A'$ a morphism in $\A/\X$. Assume that $\ul{f}$ is induced by the $\X$-complete conflation $A\xrto{f} B\xrto{g} C $ in $\A$. Then we have a pushout diagram in $\A$ by \cite[Proposition 2.12]{Buhler10}:	
		\[\xy\xymatrixcolsep{2pc}\xymatrix@C14pt@R14pt{
A\ar[r]^f\ar[d]_\alpha&B\ar[r]^g\ar[d]^\beta&C\ar@{=}[d]\\
A'\ar[r]^{f'}&B'\ar[r]^{g'}&C}
\endxy\]
Since $\X$-monics are closed under pushout we know that $f'$ is an $\X$-monic. Since $g=g'\circ \beta$ is an $\X$-epic, so is $g'$. Thus the second row in the above diagram is $\X$-complete, from which we know that $\ul{f}'$ is an inflation in $\A/\X$. Note that the conflation $ A\xrto{\left(\begin{smallmatrix} f\\
		\alpha \end{smallmatrix}\right)} B\oplus A'\xrightarrow{(\beta, -f')} B' $ in $\A$ is also $\X$-complete, so it induces a sequence in $\ul{\E}_\X$: $A\xrto{\left(\begin{smallmatrix} \ul{f}\\
		\ul{\alpha} \end{smallmatrix}\right)} B\oplus A'\xrightarrow{(\ul{\beta}, -\ul{f}')} B'$  which lifts to a pushout diagram in $\A/\X$. \end{proof}
\begin{remark} \label{Frobenius} By the proof of Theorem \ref{thm:sexact}, we know that if $\A$ has a Frobenius exact structure $\mathcal{E}'\subset \mathcal{E}$ such that each object in $\X$ is projective-injective with respect to $\mathcal{E}'$, then $\A/\X$ also has a Frobenius exact structure induced by $\mathcal{E}'$.
\end{remark}

\begin{example}
	Let $k$ be a field and $p\geq 2$ a natural number. Let $\mathbb{X}$ be the weighted projective line of type $(2, 3, p)$. Let ${\rm vec}$-$\mathbb{X}$ be the category of vector bundles and $\F$ the additive closure of the {\it fading} line bundles in the sense of \cite{Kussin/Lenzing/Meltz}. Then ${\rm vec}\mbox{-}\mathbb{X}$ has an exact structure $\mathcal{E}$ and $\F$ satisfies the conditions of Theorem \ref{thm:sexact} by the proof of \cite[Proposition 4.13 ]{Kussin/Lenzing/Meltz}. Then $\mathrm{vet}\mbox{-}\mathbb{X}/\F$ is a preabelian category and has an exact structure induced by $\F$-complete conflations. In particular, since $\mathrm{vec}\mbox{-}\mathbb{X}$ has a Frobenius exact structure given by the {\it distinguished} conflations in $\mathcal{E}$ with line bundles as the projective-injective objects \cite{Kussin/Lenzing/Meltz}, then $\mathrm{vec}\mbox{-}\mathbb{X}/\F$ has a Frobenius exact structure induced by the distinguished conflations, this is \cite[Theorem A (2)]{Kussin/Lenzing/Meltz}.
\end{example}

\subsection*{Abelian subfactor categories}

The following result extends \cite[Theorem 3.3]{Koenig/Zhu}:
\begin{theorem} \label{thm:sabelian}
	Let $(\T, [1], \Delta)$ be a triangulated category and $\X$ an additive subcategory of $\T$. Consider the following conditions:

	$(\mathrm{a})$ \ For each $A\in \T$, there are triangles $A[-1]\to X_1\to X_0\xrto{p} A$ and $A\xrto{i} X^0\to X^1\to A[1]$ in $\Delta$ such that $p$ is an $\X$-precover, $i$ is an $\X$-preenvelope and $X^1, X_1\in \X$.
	
	$(\mathrm{b})$ \ For any triangle $A\xrto{f} B\xrto{g} C\to A[1]$ in $\Delta$, $f$ is an $\X$-monic if $\ul{g}$ is an epimorphism in $\T/\X$ and $g$ is an $\X$-epic if $\ul{f}$ is a monomorphism in $\T/\X$.
	
	\noindent	If condition $(\mathrm{a})$ holds, then $\T/\X$ is a preabelian category.  Furthermore, if condition $(\mathrm{b})$ holds also, then $\T/\X$ is an abelian category.
\end{theorem}
\begin{proof} \ Assume that the condition (a) holds. Then  by Proposition \ref{prop:prtc}, $(\T, [1], \Delta, \X)$ is a partial right triangulated category, and by Proposition \ref{prop:pltc} $(\T, [-1], \Delta, \X)$ is a partial left triangulated category. Thus $\T/\X$ has a right triangulated structure $(\Sigma^\X, \Delta^\X)$ and a left triangulated structure $(\Omega_\X, \nabla_\X)$ by Theorem \ref{thm:main}. By the constructions of  $\Sigma^\X$ and $\Omega_\X$,  we know that they are zero functors. Therefore, any morphism in $\T/\X$ has a kernel and a cokernel, i.e. $\T/\X$ is a preabelian category.
	
	Now suppose that condition (b) holds also. If $\ul{f}\colon A\to B$ is a monomorphism in $\T/\X$, then there is a right triangle $A\xrto{\ul{f}} B\xrto{\ul{g}} C\to 0$ in $\Delta^\X$. Without loss of generality, we may assume that this triangle is a standard right triangle, i.e. it is induced by a triangle $A\xrto{f} B\xrto{g} C\stackrel{h}\to A[1]$ in $\Delta$. By assumption, $g$ is an $\X$-epic. Thus by the rotation axiom, $ C[-1]\to A\xrto{f} B\xrto{g} C$ is in $\Delta$ and it induces a standard left triangle $0\to A \xrto{\ul{f}} B\xrto{\ul{g}} C$ in $\nabla_\X$. Thus $\ul{f}$ is a kernel of $\ul{g}$. Dually we can show that any epimorphism is a cokernel. So $\T/\X$ is an abelian category.
\end{proof}
\begin{example}
	If $\X$ is a {\it tilting subcategory} of $\T$ in the sense of \cite[Definition 3.1]{Koenig/Zhu}, then $\X$ satisfies the conditions (a) and (b) of Theorem \ref{thm:sabelian} by \cite[Lemma 3.2 and Theorem 2.3]{Koenig/Zhu}. Thus $\T/\X$ is an abelian category which is \cite[Theorem 3.3]{Koenig/Zhu}.
\end{example}

\section{Triangulated subfactor categories}
In this section we introduce the notion of a partial triangulated category, construct triangulated subfactor categories and give a model structure of Iyama-Yoshino triangulated subfactor categories.

\subsection*{Triangulated structures}
Let $(\A, \Sigma, \R(\C), \X)$ be a partial right triangulated category. By Theorem \ref{thm:main}, the subfactor category $\C/\X$ has a right triangulated structure $(\Sigma^\X, \Delta^\X)$. It is natural to ask when this is a triangulated structure on $\C/\X$. This question naturally suggests the following definition.

\begin{definition} \label{defn:Frobeniusprtc} \ The partial right triangulated category $(\A, \Sigma, \R(\C), \X)$ is called {\it Frobenius} if

(i)  For each $A\in \C$, there is a right $\C$-triangle $K\xrto{u} X\xrto{v} A\to \Sigma(K)$ with $u$ an $\X$-preenvelope.

(ii) For each $A\in \C$, $p^A$ is an $\X$-precover in the fixed right $\C$-triangle $A\xrto{i^A} X^A\xrto{p^A} U^A\xrto{q^A} \Sigma(A)$.

(iii)  For any commutative diagram of the fixed right $\C$-triangles
\[\xy\xymatrixcolsep{2pc}\xymatrix@C14pt@R14pt{
A\ar[r]^-{i^A}\ar@{.>}[d]_{\alpha}&X^A\ar[r]^-{p^A}\ar[d]^\beta &U^A\ar[r]^-{q^A}\ar[d]^\gamma& \Sigma(A)\ar@{.>}[d]^{\Sigma(\alpha)}\\
B\ar[r]^{i^B}&X^B\ar[r]^-{p^B}&U^B\ar[r]^-{q^B}&\Sigma(B)}
\endxy\]
there is a morphism $\alpha\colon A\to B$ such that $\beta\circ i^A=i^B\circ \alpha$ and $\Sigma(\alpha)\circ q^A= q^B\circ \gamma$, and if $\gamma$ factors through $p^B$ then $\alpha$ factors through $i^A$.
\end{definition}

Dually, we can define the notion of a {\it Frobenius} partial left triangulated category.

\begin{proposition} \label{prop:Frobeniusprtc} $(\mathrm{i})$ \ If the partial right triangulated category $(\A, \Sigma, \R(\C), \X)$ is Frobenius, then $\C/\X$ is a triangulated category.

$(\mathrm{ii})$ \ If the partial left triangulated category $(\A, \Omega, \rmL(\C), \X)$ is Frobenius, then $\C/\X$ is a triangulated category.
\end{proposition}

\begin{proof} \ We only prove (i), the statement (ii) can be proved dually.

Since $\C/\X$ has a right triangulated structure $(\Sigma^\X, \Delta^\X)$ as constructed in Theorem \ref{thm:main} induced by $\R(\C)$, so we only need to show that $\Sigma^\X$ is an equivalence which is equivalent to prove that $\Sigma^\X$ is dense and fully faithful by \cite[Theorem II.2.7]{Gelfand-Manin}.

We first show that $\Sigma^\X$ is dense. In fact, given any $A\in \C/\X$, by Definition \ref{defn:Frobeniusprtc} (i), there is a right $\C$-triangle $K\xrto{u} X\xrto{v} A\to \Sigma(K)$ such that $u$ is an $\X$-preevelope. By the construction of $\Sigma^\X$ and Lemma \ref{lemma:uniqueness} (ii), $A\cong \Sigma^\X(K)$ in $\C/\X$, so $\Sigma^\X$ is dense.

To see that $\Sigma^\X$ is fully faithful, let $\ul{\gamma}\colon \Sigma^\X(A)\to \Sigma^\X(B)$ be a morphism in $\C/\X$. By the construction of $\Sigma^\X$, $\Sigma^\X(A)=U^A$ and $\Sigma^\X(B)=U^B$. Then there exists a commutative diagram of right $\C$-triangles
	\[\xy\xymatrixcolsep{2pc}\xymatrix@C36pt@R14pt{
A\ar[r]^-{i^A}\ar[d]_{\alpha}&X^A\ar[d]^{\beta}\ar[r]^-{p^A}&U^A\ar[d]^\gamma\ar[r]^-{q^A}&\Sigma(A)\ar[d]^{\Sigma(\alpha)}\\
B\ar[r]^-{i^B}&X^B\ar[r]^-{p^B}&U^B\ar[r]^-{q^B}& \Sigma(B)}
\endxy\]
where the existence of $\beta$ is since $p^B$ is an $\X$-preenvelope and the existence of $\alpha$ is by Definition \ref{defn:Frobeniusprtc}(iii). Thus $\ul{\gamma}=\Sigma^\X(\ul{\alpha})$. Moreover, if $\ul{\gamma}=0$, i.e., $\gamma$ factors through some object in $\X$, then it factors through $p^B$ since $p^B$ is an $\X$-preenvelope. Thus $\alpha$ factors through $i^A$ by Definition \ref{defn:Frobeniusprtc} (iii). So $\ul{\alpha}=0$ in $\C/\X$ and then $\Sigma^\X$ is fully faithful.
\end{proof}

\subsection*{Partial triangulated categories}
\begin{definition} \ \label{defn:ptc}\ Let $(\A, \Omega, \rmL(\C),\X)$ be a partial left triangulated category and $(\A, \Sigma, \R(\C), \X)$ a partial right triangulated category. The six-tuple $(\A, \Omega, \Sigma, \rmL(\C), \R(\C), \X)$ is called a {\it partial triangulated category} if

(a) \ $(\Omega, \Sigma)$ is an adjoint pair with $\psi$ as the adjunction isomorphism.

(b) \ If $\Omega(C)\xrto{u} A\xrto{v} B\xrto{w} C$ is in $\rmL(\C)$ with $B,C\in\C$, then $A\xrto{v} B\xrto{w} C\xrto{-\psi_{C, A}(u)} \Sigma(A)$ is in $\R(\C)$. Conversely, if $A\xrto{f} B\xrto{g} C\xrto{h} \Sigma(A)$ is in $\R(\C)$ with $A,B\in \C$, then $\Omega(C)\xrto{-\psi_{C,A}^{-1}(h)} A\xrto{f} B\xrto{g} C$ is in $\rmL(\C)$.

(c) \ For each $A\in \C$, $p^A$ is an $\X$-precover in the fixed right $\C$-triangle $A\xrto{i^A} X^A\xrto{p^A} U^A\xrto{q^A} \Sigma(A)$ and $\iota_A$ is an $\X$-preenvelope in the fixed left $\C$-triangle $\Omega(A)\xrto{\nu_A} K_A\xrto{\iota_A} X_A\xrto{\pi_A} A$.
\end{definition}

\begin{example} \label{exam:ptc} (i) \  Let $\F$ be a Frobenius category. Let $\I$ be the subcategory of projective-injective objects of $\F$. Then $(\F, 0,0, \rmL(\F), \R(\F), \I)$ is a partial triangulated category, where $\rmL(\F)$ and $\R(\F)$ are as constructed in Proposition \ref{prop:pltcexact}.

(ii) \  Let $(\mathcal{T}, [1])$ be a triangulated category with $\X\subseteq \mathcal{C}$ two additive subcategories closed under direct summands. Assume that $(\C, \C)$ forms an $\X$-mutation pair, then $(\T, [-1], [1], \rmL(\C), \R(\C), \X)$ is a partial triangulated category in which $\rmL(\C)$ and $\R(\C)$ are as constructed in Example \ref{exam:tricat} (i).

(iii) \ Let $(\T, [1])$ be a triangulated category. Let $\C$ and $\X$ be two additive categories such that $\C$ is $\X$-Frobenius in the sense of \cite[Definition 3.2]{Beligiannis13}. Then $(\T, [-1], [1], \rmL(\C), \R(\C), \X)$ is a partial triangulated category (see the proof of \cite[Proposition 3.1]{Beligiannis13}) in which $\rmL(\C)$ and $\R(\C)$ are as constructed in Example \ref{exam:tricat} (ii).

(iv) \ Let $k$ be an algebraically closed field and let $(\T, [1], \Delta)$ be a $k$-linear triangulated category with Serr functor $S$, which has finite dimensional Hom spaces and split idempotents. Let $\X$ be an additive subcategory of $\T$ such that for each object $A$ in $\T$ there are triangles $A\xrto{i} X\to C\to A[1]$ and $A[-1]\to K\to X'\xrto{\pi} A$ in $\Delta$ with $i$ an $\X$-preenvelope and $\pi$ an $\X$-precover. Denote by $\tau=[-1]\circ S$ the Auslander-Reiten translation. If $\X$ satisfies $\tau(\X)=\X$ (i.e. the set of objects isomorphic to objects in $\tau(\X)$ is equal to $\X$), then $(\T, [-1], [1], \Delta, \Delta, \X)$ is a partial triangulated category by \cite[Lemma 3.2]{Jorgensen}.
\end{example}

\begin{proposition} \label{prop:ptc} \ Let $(\A, \Omega, \Sigma, \rmL(\C), \R(\C), \X)$ be a partial triangulated category. Then both the partial right triangulated category $(\A, \Sigma, \R(\C), \X)$ and the partial left triangulated category $(\A, \Omega, \rmL(\C), \X)$ are Frobenius.
\end{proposition}

\begin{proof} \ We only prove that $(\A, \Sigma, \R(\C), \X)$ is Frobenius since the other case can be proved dually. We verify Definition \ref{defn:Frobeniusprtc} (i)-(iii) one by one.

(i) \ For each object $A$ in $\C$, there is a fixed left $\C$-triangle $\Omega(A)\xrto{\nu_A} K_A\xrto{\iota_A} X_A\xrto{\pi_A} A$ with $\pi_A$ an $\X$-precover. Since $(\A, \Omega, \Sigma, \rmL(\C), \R(\C), \X)$ is a partial triangulated category, $K_A\xrto{\iota_A} X_A\xrto{\pi_A} A\xrto{-\psi_{A, K_A}(\nu_A)} \Sigma(K_A)$ is a right $\C$-triangle with $\iota_A$ an $\X$-preenvelope.

(ii) For each $A\in \C$, $p^A$ is an $\X$-precover in the fixed right $\C$-triangle $A\xrto{i^A} X^A\xrto{p^A} U^A\xrto{q^A} \Sigma(A)$ since $\Omega(U^A)\xrto{-\psi_{U^A, A}^{-1}(q^A)}A\xrto{i^A} X^A\xrto{p^A} U^A$ is in $\rmL(\C)$ with $p^A$ an $\X$-precover by assumption.

(iii) Consider the following commutative diagram of the fixed right $\C$-triangles
\[\xy\xymatrixcolsep{2pc}\xymatrix@C14pt@R14pt{
A\ar[r]^-{i^A}&X^A\ar[r]^-{p^A}\ar[d]^\beta &U^A\ar[r]^-{q^A}\ar[d]^\gamma& \Sigma(A)\\
B\ar[r]^{i^B}&X^B\ar[r]^-{p^B}&U^B\ar[r]^-{q^B}&\Sigma(B)}
\endxy\]
By assumption, we have a commutative diagram of left $\C$-triangles:
	\[\xy\xymatrixcolsep{2pc}\xymatrix@C36pt@R14pt{
\Omega(U^A)\ar[r]^-{-\psi_{U^A,A}^{-1}(q^A)}\ar[d]_{\Omega(\gamma)}&A\ar[r]^-{i^A}\ar[d]^\alpha&X^A\ar[d]^{\beta}\ar[r]^-{p^A}&U^A\ar[d]^\gamma\\
\Omega(U^B)\ar[r]^-{-\psi_{U^B,B}^{-1}(q^B)}&B\ar[r]^-{i^B}&X^B\ar[r]^-{p^B}&U^B}
\endxy\]
where the existence of $\alpha$ is by the axiom (PLT3) of a partial left triangulated category. While by the axiom (PLT2), if $\gamma$ factors through $q^B$, then $\alpha$ factors through $i^A$. We claim that $q^B\circ \gamma=\Sigma(\alpha)\circ q^A$. In fact, follow from the naturality of $\psi_{U^B, B}^{-1}$ in $U^B$, we have the following commutative diagram
\[\xy\xymatrixcolsep{2pc}\xymatrix@C18pt@R14pt{
\Hom_\A(U^B, \Sigma(B))\ar[r]^-{\psi_{U^B,B}^{-1}}\ar[d]_{\gamma^*}&\Hom_\A(\Omega(U^B), B)\ar[d]^{\Omega(\gamma)^*}\\
\Hom_\A(U^A, \Sigma(B))\ar[r]^-{\psi_{U^A,B}^{-1}}&\Hom_\A(\Omega(U^A),B)}
\endxy\]
and then for $q^B\colon U^B\to \Sigma(B)$ we get $$\psi^{-1}_{U^B, B}(q^B)\circ \Omega(\gamma)=\psi^{-1}_{U^A, B}(q^B\circ \gamma).$$
Similarly, by the naturality of $\psi^{-1}_{U^A, A}$ in $A$, we obtain the following commutative diagram
 \[\xy\xymatrixcolsep{2pc}\xymatrix@C18pt@R14pt{
\Hom_\A(U^A, \Sigma(A))\ar[r]^-{\psi_{U^A,A}^{-1}}\ar[d]_{\Sigma(\alpha)_*}&\Hom_\A(\Omega(U^A), A)\ar[d]^{\alpha_*}\\
\Hom_\A(U^A, \Sigma(B))\ar[r]^-{\psi_{U^A,B}^{-1}}&\Hom_\A(\Omega(U^A),B)}
\endxy\]
This yields for $q^A\colon U^A\to \Sigma(A)$ the formula $$\psi^{-1}_{U^A, B}(\Sigma(\alpha)\circ q^A)=\alpha\circ \psi^{-1}_{U^A, A}(q^A).$$
Since $\psi^{-1}_{U^B, B}(q^B)\circ \Omega(\gamma)=\alpha\circ \psi^{-1}_{U^A, A}(q^A)$ we know that $\psi^{-1}_{U^A, B}(\Sigma(\alpha)\circ q^A)=\psi^{-1}_{U^A, B}(q^B\circ \gamma)$ and thus $q^B\circ \gamma=\Sigma(\alpha)\circ q^A.$
So we have the following commutative diagram:
\[\xy\xymatrixcolsep{2pc}\xymatrix@C16pt@R14pt{
A\ar[r]^-{i^A}\ar[d]_{\alpha}&X^A\ar[r]^-{p^A}\ar[d]^{\beta}&U^A\ar[r]^-{q^A} \ar[d]^\gamma&\Sigma(A)\ar[d]^{\Sigma(\alpha)}\\
B\ar[r]^-{i^B}&X^B\ar[r]^-{p^B}&U^B\ar[r]^-{q^B}&\Sigma(B)}
\endxy\]
\end{proof}

\begin{remark} In general, even a Frobenius partial right triangulated category $(\A, \Sigma, \R(\C), \X)$ and a Frobenius partial left triangulated category $(\A, \Omega, \rmL(\C), \X)$ share the same underlying categories, the triangulated structures $(\Sigma^\X, \Delta^\X)$ induced by $\R(\C)$ and $(\Omega_\X, \nabla_\X)$ induced by $\rmL(\C)$ on $\C/\X$ are not necessarily the same. For example, let $\T$ be an additive category and assume that it has two different triangulated structures
$([1], \Delta)$ and $([1], \Delta')$ with the same shift functor. Let $\R(\T)=\Delta$ and $\rmL(\T)=\Delta'$. Then $(\T, [1], \R(\T), 0)$ is a Frobenius partial right triangulated category and $(\T, [-1], \rmL(\T), 0)$ is a Frobenius partial triangulated category. However, the triangulated structure on $\T=\T/0$ induced by $\R(\T)$ is $([1], \Delta)$ and the triangulated structure on $\T=\T/0$ induced by $\rmL(\T)$ is $([-1], \Delta')$, and they are different.
\end{remark}
We have the following result for a partial triangulated category.
\begin{theorem} \label{thm:ptc}  Let $(\A, \Omega, \Sigma, \rmL(\C), \R(\C), \X)$ be a partial triangulated category. Then the subfactor $\C/\X$ is a triangulated category and the two triangulated structures induced by $\rmL(\C)$ and $\R(\C)$ respectively are the same.
\end{theorem}
\begin{proof} By Proposition \ref{prop:Frobeniusprtc} and Proposition \ref{prop:ptc}, we know that the subcactor $\C/\X$ has a triangulated structure $(\Sigma^\X, \Delta^\X)$ induced by $\R(\C)$ and a triangulated structure $\rmL(\C)$.  Now we show that the triangulated structures $(\Sigma^\X, \Delta^\X)$ and $(\Omega_\X, \nabla_\X)$ are the same by showing that $(\Sigma^\X, \Omega_\X)$ is an adjoint pair and $\Delta^\X=\nabla_\X$.

We first show that $(\Sigma^\X, \Omega_\X)$ is an adjoint pair. For $A, B\in \C$, we define a map $$\varphi_{A,B}\colon \Hom_{\C/\X}(\Sigma^\X(A), B)\to \Hom_{\C/\X}(A, \Omega_\X(B))$$ by sending $\ul{f}$ to $-\ul{g}$, where $g$ fits into the following commutative diagram of left $\C$-triangles
\[\xy\xymatrixcolsep{2pc}\xymatrix@C36pt@R14pt{
\Omega(U^A)\ar[r]^-{-\psi_{U^A,A}^{-1}(q^A)}\ar[d]_{\Omega(f)}&A\ar[r]^-{i^A}\ar[d]^{g}&X^A\ar[r]^-{p^A}\ar[d]^{\sigma}& U^A\ar[d]^f\\
\Omega(B)\ar[r]^-{\nu_B}&U_B\ar[r]^-{\iota_B}& X_B\ar[r]^{\pi_B}&B}
\endxy\]
obtained by the dual of Lemma \ref{lemma:notation}. By the proof of Proposition \ref{prop:ptc}, we have $-\Sigma(g)\circ q^A=\psi_{B, U_B}(\nu_B)\circ f$ since $-g\circ \psi_{U^A, A}^{-1}(q^A)=\nu_B\circ \Omega(f)$. So we have the following commutative diagram of right $\C$-triangles
\[\xy\xymatrixcolsep{2pc}\xymatrix@C38pt@R14pt{
A\ar[r]^-{i^A}\ar[d]_{g}&X^A\ar[r]^-{p^A}\ar[d]^{\sigma}& U^A\ar[d]^f \ar[r]^-{q^A}&\Sigma(A)\ar[d]^{\Sigma(g)}\\
U_B\ar[r]^-{\iota_B}& X_B\ar[r]^{\pi_B}&B\ar[r]^-{-\psi_{B,U_B}(\nu_A)}&\Sigma(U_B)}
\endxy\]
Therefore $\ul{f}$ and $\ul{g}$ are determined mutually and then $\varphi_{A,B}$ is a bijection by Lemma \ref{lemma:notation} and its dual. The naturality of $\varphi_{A,B}$ on $A$ and $B$ can be verified directly. So $(\Sigma^\X, \Omega_\X)$ is an adjoint pair.

Assume that we have a standard triangle $A\xrto{\ul{f}}B\xrto{-\ul{\eta}}N\xrto{\ul{\zeta}}\Sigma^\X(A)$ in $\Delta^\X$ which is induced in (\ref{pullback}) by the right $\C$-sequence $A\xrto{\left(\begin{smallmatrix}
i^A \\
f
\end{smallmatrix}\right)}X^A\oplus B\xrto{(\theta, -\eta)}N\to \Sigma(A)$. We will show that it is in $\nabla_\X$.  Since $(\Omega_\X, \nabla_\X)$ is also a triangulated structure on $\C/\X$, there is a triangle $\Omega_\X(\Sigma^\X(A))\xrto{\ul{u}}K\xrto{\ul{v}}N\xrto{\ul{\zeta}} \Sigma^\X(A)$ in $\nabla_\X$ and we may assume that it is induced by the following commutative diagram of left $\C$-triangles:
\[\xy\xymatrixcolsep{2pc}\xymatrix@C32pt@R18pt{
\Omega(\Sigma^\X(A))\ar[r]^-{\nu_{\Sigma^\X(A)}}\ar@{=}[d]&\Omega_\X(\Sigma^\X(A))\ar[r]^-{\iota_{\Sigma^\X(A)}} \ar[d]^u&X_{\Sigma^\X(A)}\ar[r]^-{\pi_{\Sigma^\X(A)}}\ar[d]^{\left(\begin{smallmatrix}
1 \\
0
\end{smallmatrix}\right)}& \Sigma^\X(A)\ar@{=}[d]\\
\Omega(\Sigma^\X(A))\ar[r]^-w&K\ar[r]^-{\left(\begin{smallmatrix}
\vartheta \\
v
\end{smallmatrix}\right)}&X_{\Sigma^\X(A)}\oplus N\ar[r]^-{(\pi_{\Sigma^\X(A)}, \zeta)}& \Sigma^\X(A)}
\endxy\]
By the diagram (\ref{pullback}), we know that $B\xrto{\eta} N\xrto{\zeta} \Sigma^\X(A)\xrto{\Sigma(f)\circ q^{A}}\Sigma(B)$ is in $\R(\C)$ and thus $\Omega(\Sigma^\X(A))\xrto{-\psi_{\Sigma^\X(A),B}^{-1}(\Sigma(f)\circ q^A)}B\xrto{\eta}N\xrto{\zeta}\Sigma^\X(A)$ is in $\rmL(\C)$.  So there is a morphism $\alpha\colon B\to K$ such that the following diagram is commutative
\[\xy\xymatrixcolsep{2pc}\xymatrix@C58pt@R18pt{
\Omega(\Sigma^\X(A))\ar@{=}[d]\ar[r]^-{-\psi_{\Sigma^\X(A),B}^{-1}(\Sigma(f)\circ q^A)} &B\ar[r]^-{\eta} \ar[d]_{\alpha}&N\ar[r]^-{\zeta}\ar[d]^{\left(\begin{smallmatrix}
0 \\
1
\end{smallmatrix}\right)}& \Sigma^\X(A)\ar@{=}[d]\\
\Omega(\Sigma^\X(A))\ar[r]^-{w}&K\ar[r]^-{\left(\begin{smallmatrix}
\vartheta \\
v
\end{smallmatrix}\right)}&X_{\Sigma^\X(A)}\oplus N\ar[r]^-{(\pi_{\Sigma^\X(A)}, \zeta)}& \Sigma^\X(A)}
\endxy\]
and then $\ul{v}\circ \ul{\alpha}=\ul{\eta}$. By the dual of Lemma \ref{lemma:notation}, we can deduce that $\ul{u}\circ [-\varphi_{A,A}(\ul{1}_{\Sigma^\X(A)})]=\ul{\alpha}\circ \ul{f}$. Thus we have the following commutative diagram
\[
\xy\xymatrixcolsep{2pc}\xymatrix@C16pt@R16pt{
A\ar[r]^-{\ul{f}}\ar[d]_{\varphi_{A,A}(\ul{1}_{\Sigma^\X(A)})}& B\ar[r]^-{-\ul{\eta}}\ar[d]_{-\ul{\alpha}}&N\ar[r]^-{\ul{\zeta}}\ar@{=}[d]& \Sigma^\X(A)\ar@{=}[d]\\
\Omega_\X(\Sigma^\X(A))\ar[r]^-{\ul{u}}& K\ar[r]^-{\ul{v}}&N\ar[r]^-{\ul{\zeta}}& \Sigma^\X(A)}
\endxy
\]
Since $\varphi_{A,A}(\ul{1}_{\Sigma^\X(A)})$ is an isomorphism and both $(\Sigma^\X, \Delta^\X)$ and $(\Omega_\X, \nabla_\X)$ are triangulated structures, we know that $\ul{\alpha}$ is an isomorphism. Therefore, $A\xrto{\ul{f}}B\xrto{-\ul{\eta}}N\xrto{\ul{\zeta}}\Sigma^\X(A)$ is in $\nabla_\X$ and then $\Delta^\X\subseteq \nabla_\X$. Dually we can show that $\nabla_\X\subseteq \Delta^\X$ and then they are equal.
\end{proof}

\subsection*{A model structure of Iyama-Yoshino triangulated subfactor categories}

Let $\C$ be an additive category. We recall the definition of a model structure on $\C$ in the sense of Quillen \cite[Page 1.1, Definition 1]{Quillen67} (see also \cite[Definition 4.1]{Beligiannis01}) which is called a {\it classical model structure} here. The reason is that the modern definition of a model structure often corresponds to what Quillen called a {\it closed model structure} (i.e. a classical model structure which satisfies the {\it retraction} axiom).
\begin{definition} A {\it classical model structure} on $\C$ consists of three classes of morphisms called {\it cofibrations, fibrations} and {\it weak equivalences}, denoted by $\mathrm{Cof}, \mathrm{Fib}$ and $\mathrm{W}$ respectively, which contain isomorphisms and satisfy the following axioms:

(M1) For any commutative diagram
\[\xy\xymatrixcolsep{2pc}\xymatrix@C14pt@R14pt{
A\ar[r]^f\ar[d]_i&B\ar[d]^p\\
C\ar[r]^g\ar@{.>}[ru]^h&D}
\endxy\]
where $i$ is a cofibration, $p$ is a fibration, and where either $i$ or $p$ is a weak equivalence, then the dotted morphism $h$ exists such that $h\circ i=f$ and $p\circ h=g$.

(M2) Any morphism $f$ in $\C$  can be factored in two ways: (i)\ $f=p\circ i$, where $i$ is a cofibration and  $p$ is a weak equivalence and fibration, and (ii)\ $f=p\circ i$, where $i$ is a weak equivalence and cofibration and $p$ is a fibration.

(M3) (Two out of three property) \ If $f, g$ are composable morphisms in $\C$ and if two of the morphisms $f, g$ and $g\circ f$ are weak equivalences, so is the third.
\end{definition}

If $(\A, \Omega, \Sigma, \rmL(\C), \R(\C), \X)$ is a partial triangulated category, then we define three classes of morphisms in $\C$ as follows:

(i) $\mathrm{Cof}(\C)=\{f\colon A\to B \ | \ \exists \  A\xrto{f} B\to C\to \Sigma(A)\in \R(\C)\}$;

(ii) $\mathrm{Fib}(\C)=\{f\colon A\to B \ | \ \exists \ \Omega(A)\to K\to A\xrto{f} B \in \rmL(\C)\}$;

(iii) $\mathrm{W}(\C)=\{f\colon A\to B \ | \ \ul{f} \ \mbox{is an isomorphism in}  \ \C/\X \}$.

Recall that an additive category $\C$ is said to be {\it weakly idempotent complete} if every split monomorphism has a cokernel in $\C$. We have

\begin{proposition}  \label{prop:model} Let $(\A, \Sigma, \rmL(\C), \R(\C), \X)$ be a partial triangulated category. If $\C$ is weakly idempotent complete and $\X$ is closed under direct summands, then

$(\mathrm {i})$\ $(\mathrm{Cof}(\C), \mathrm{Fib}(\C), \mathrm{W}(\C))$ is a model structure on $\C$.

 $(\mathrm{ii})$\ The homotopy category of $(\mathrm{Cof}(\C), \mathrm{Fib}(\C), \mathrm{W}(\C))$ is $\C/\X$ which is a triangulated category.
\end{proposition}
\begin{proof} The proof of \cite[Theorem 4.5]{Beligiannis01} also works here, and by Theorem \ref{thm:ptc} we know that $\C/\X$ is a triangulated category. \end{proof}

\begin{corollary} \ \label{cor:IYmodel} Let $\T$ be a triangulated category. Let $\X\subseteq \C$ be two additive subcategories of $\T$ closed under direct summands. If $(\mathcal{C}, \mathcal{C})$ forms an $\mathcal{X}$-mutation pair, then there is a model structure on $\C$ such that the Iyama-Yoshino triangulated subfactor category $\C/\X$ is the corresponding homotopy category.
\end{corollary}
\begin{proof} \  Since $\T$ is a triangulated category and $\C$ is closed under direct summands, we know that $\C$ is weakly idempotent complete. Thus the assertion follows from Proposition \ref{prop:model}. \end{proof}

\end{document}